\documentclass[10pt,a4paper]{article}
\usepackage{amsmath}
\usepackage{amsthm}
\usepackage{amsfonts}
\usepackage{multirow}
\usepackage[table,xcdraw]{xcolor}
\usepackage{subcaption}
\usepackage{float}
\usepackage{algorithm}
\usepackage{algpseudocode}
\usepackage{hyperref}
\usepackage{tabularx}
\usepackage{amssymb}
\usepackage{graphicx}
\usepackage{tcolorbox}
\usepackage{booktabs}
\usepackage[table,xcdraw]{xcolor}
\usepackage{multirow}
\usepackage{rotating,tabularx}

\usepackage[left=2cm,right=2cm,top=2cm,bottom=2cm]{geometry}
\newcolumntype{L}[1]{>{\raggedright\let\newline\\\arraybackslash\hspace{0pt}}m{#1}}
\newcolumntype{C}[1]{>{\centering\let\newline\\\arraybackslash\hspace{0pt}}m{#1}}
\newcolumntype{R}[1]{>{\raggedleft\let\newline\\\arraybackslash\hspace{0pt}}m{#1}}

\newtheorem{Theorem}{Theorem}[section]
\newtheorem{Proposition}[Theorem]{Proposition}
\newtheorem{Remark}[Theorem]{Remark}
\newtheorem{Lemma}[Theorem]{Lemma}

\newtheorem{Definition}[Theorem]{Definition}

\usepackage{hyperref}
\hypersetup{colorlinks=true,urlcolor=blue}
\expandafter\let\expandafter\oldproof\csname\string\proof\endcsname
\let\oldendproof\endproof
\renewenvironment{proof}[1][\proofname]{
\oldproof[\ttfamily\scshape \bf #1.]
}{\oldendproof}
\def\ve{\varepsilon}

\def\emp{\emptyset}

\def\dom{{\rm dom}\,}

\def\epi{{\rm epi\,}}

\def\min{\mbox{\rm minimize}}

\def\B{\mathbb B}

\def\ox{\overline{x}}
\def\oy{\overline{y}}
\def\oz{\overline{z}}

\def\disp{\displaystyle}
\def\tto{\rightrightarrows}

\def\Tilde{\widetilde}
\def\Bar{\overline}
\def\ra{\rangle}
\def\la{\langle}
\def\ve{\varepsilon}
\def\epsilon{\varepsilon}
\def\ox{\bar{x}}
\def\oy{\bar{y}}
\def\oz{\bar{z}}

\def\ov{\bar{v}}

\def\ou{\bar{u}}

\def\co{\mbox{\rm co}\,}

\def\gph{\mbox{\rm gph}\,}
\def\epi{\mbox{\rm epi}\,}

\def\dom{\mbox{\rm dom}\,}

\def\dn{\downarrow}
\def\O{\Omega}
\def\ph{\varphi}
\def\emp{\emptyset}
\def\st{\stackrel}
\def\oR{\Bar{\R}}

\def\gg{\gamma}

\def\al{\alpha}

\def \N{{\rm I\!N}}
\def \R{{\rm I\!R}}

\def\vt{\vartheta}

\def\Limsup{\mathop{{\rm Lim}\,{\rm sup}}}

\def\Limsup{\mathop{{\rm Lim}\,{\rm sup}}}

\numberwithin{equation}{section}

\title{\bf Globally Convergent Coderivative-Based Generalized Newton Methods in Nonsmooth Optimization}
\author{Pham Duy Khanh\footnote{Department of Mathematics, Ho Chi Minh City University of Education, Ho Chi Minh City, Vietnam. E-mails: pdkhanh182@gmail.com, khanhpd@hcmue.edu.vn.}\quad Boris S. Mordukhovich\footnote{Department of Mathematics, Wayne State University, Detroit, Michigan, USA. E-mail: aa1086@wayne.edu.}\quad Vo Thanh Phat\footnote{Department of Mathematics, Wayne State University, Detroit, Michigan, USA. E-mail: phatvt@wayne.edu.} \quad Dat Ba Tran\footnote{Department of Mathematics, Wayne State University, Detroit, Michigan, USA. E-mail: tranbadat@wayne.edu.}}
\begin{document}
\maketitle

\noindent
{\small{\bf Abstract}. This paper proposes and justifies two globally convergent Newton-type methods to solve unconstrained and constrained problems of nonsmooth optimization by using tools of variational analysis and generalized differentiation. Both methods are coderivative-based and employ generalized Hessians (coderivatives of subgradient mappings) associated with objective functions, which are either of class ${\cal C}^{1,1}$, or are represented in the form of convex composite optimization, where one of the terms may be extended-real-valued. The proposed globally convergent algorithms are of two types. The first one extends the damped Newton method and requires positive-definiteness of the generalized Hessians for its well-posedness and efficient performance, while the other algorithm is of  {the regularized Newton  type} being well-defined when the generalized Hessians are merely positive-semidefinite. The obtained convergence rates for both methods are at least linear, but become superlinear under the semismooth$^*$ property of subgradient mappings. Problems of convex composite optimization are investigated with and without the strong convexity assumption  {on  smooth parts} of objective functions by implementing the machinery of forward-backward envelopes. Numerical experiments are conducted for Lasso problems  and for box constrained quadratic programs   with providing performance comparisons of the new algorithms and some other first-order and second-order methods that are highly recognized in nonsmooth optimization.\\[1ex]
}\\[1ex]

\section{Introduction}\label{intro}

It has been well recognized that the classical Newton method furnishes a highly efficient algorithm to solve unconstrained optimization problems of the type
\begin{equation}\label{opproblem}
\min\;\varphi(x)\;\text{ subject to }\;x\in\R^n
\end{equation}
with ${\cal C}^2$-smooth objective functions $\ph$, provided that the Hessian matrix $\nabla^2\ph(\ox)$ is positive-definite at the reference solution $\ox$ and the starting point $x^0$ is chosen sufficiently close to $\ox$. In this case, the Newton iterations exhibit the local convergence with a quadratic rate; see, e.g., \cite{BeckAl,JPang,pol}.

{To achieve the global convergence of the Newton method, various line search algorithms are implemented by using iterative procedures given in the form}
\begin{equation}\label{linealgorithm}
x^{k+1}:=x^k+\tau_k d^k\quad\text{for all }\;k\in\N:=\{1,2,\ldots\}
\end{equation}
with a stepsize $\tau_k\ge 0$ and a search direction $d^k\ne 0$. For Newton-type methods, the search directions are chosen by solving the linear equations
\begin{equation}\label{Newtonequa}
-\nabla\varphi(x^k)=H_kd^k,
\end{equation}
where $H_k:=\nabla^2\varphi(x^k)$ in the classical case, while $H_k$ is an appropriate approximation of the Hessian for various {\em quasi-Newton methods}. An efficient way to choose $H_k$ is provided by the \textit{BFGS $($Broyden-Fletcher-Goldfarb-Shanno$)$ method}; see \cite{DM,JPang,Solo14} for more details on this and related algorithms. If $H_k=\nabla^2\varphi(x^k)$ is positive-definite, algorithm \eqref{linealgorithm} with the {\em backtracking line search} is called the {\em damped Newton method} \cite{BeckAl,Boyd} to distinguish it from the {\em pure Newton method}, which uses a fixed stepsize. When $\nabla^2\varphi(x^k)$ is merely positive-semidefinite, $H_k$ in \eqref{Newtonequa} is often taken as the regularized Hessian $H_k:=\nabla^2\varphi(x^k)+\mu_k I$ with the sequence $\{\mu_k\}$ being chosen as $\mu_k:= c\|\nabla\varphi(x^k)\|$ for some constant $c>0$. The corresponding algorithm is called {the \textit{regularized Newton method}}. We refer the reader to \cite{DYF,LFQY,YF} for many interesting results in this direction.

Among the most popular Newton-type methods to solve problems of nonsmooth optimization \eqref{opproblem} with objective functions of class ${\cal C}^{1,1}$ (i.e., continuously differentiable with Lipschitzian gradients) is the {\em semismooth Newton method}. The literature on this method and its modifications is enormous; the reader is referred to the books \cite{JPang,Solo14,Klatte} and the bibliographies therein for various developments and historical remarks. In fact, most of the known results address solving the equations $f(x)=0$ with Lipschitzian vector functions $f$, as well as their generalized versions, to which optimization problems are reduced via {stationarity conditions} (observe that this is not the case of our paper). The main idea behind the semismooth Newton method is the usage of Clarke's {\em generalized Jacobian} of Lipschitzian mappings. In this way, local convergence results, together with some globalization procedures, were obtained for this method under the {\em nonsingularity} of generalized Jacobians. The reader is referred to, e.g., \cite{hint,Ul} for infinite-dimensional versions of the semismooth Newton method with applications to optimization and control problems governed by partial differential equations. Other versions of Newton-type methods to solve nonsmooth equations, generalized equations, optimization and variational problems can be found in \cite{Bonnans,smir,Donchev09,JPang,hkmp,Solo14,Klatte,nest,rob} among other publications. We are not in a position here to review numerous contributions to Newtonian methods that are not directly related to our paper; see more commentaries below concerning publications related to our results.\vspace*{0.02in}

This paper develops two {\em globally convergent} generalized Newton algorithms to solve optimization problems \eqref{opproblem} starting with the case of ${\cal C}^{1,1}$ objective functions (i.e., those being second-order nonsmooth) and then considering problems of {\em convex composite optimization} with objectives represented as sums of two convex functions such that one of them is smooth, while the other one may be extended-real-valued, which allows us to include problems of constrained optimization. The developed algorithms constitute the coderivative-based {\em generalized damped Newton method} (GDNM) and the coderivative-based {\em  {generalized regularized Newton method}}  {(GRNM)} for the classes of problems under consideration.

Roughly speaking, the major feature of both generalized Newton methods developed here is the replacement of the classical Hessian $\nabla^2\ph$ of ${\cal C}^2$-smooth functions by the {\em generalized Hessian} (or {\em second-order subdifferential}) $\partial^2\ph$ of extended-real-valued, lower semicontinuous ones. This construction was introduced by Mordukhovich \cite{m92} as the {\em coderivative} of the subgradient mapping, while enjoying nowadays comprehensive {\em calculus rules} and constructive {\em computations} for major classes of functions that naturally appear in variational analysis, optimization, and optimal control; see Section~\ref{sec:pre} for more details and references.   Due to such massive developments in variational analysis and its applications, coderivative-based algorithms deserve a strong attention in numerical optimization. This largely motivates our current study.

 Note that coderivatives have been recently employed by Gfrerer and Outrata \cite{Helmut} to design a {\em pure} Newton-type algorithm of solving {\em generalized equations} and--very differently--by Mordukhovich and Sarabi \cite{BorisEbrahim} to find local minimizers of \eqref{opproblem} that were assumed to be {\em tilt-stable} in the sense of Poliquin and Rockafellar \cite{Poli}. The results of \cite{BorisEbrahim} were obtained first for ${\cal C}^{1,1}$ objectives and then were propagated to a general class of {\em prox-regular} functions by using Moreau envelopes. The {\em local} superlinear convergence of these Newtonian algorithms was established in \cite{Helmut,BorisEbrahim} under the {\em semismooth$^*$} assumption on the mapping in question, the property introduced in \cite{Helmut} as a less restrictive version of semismoothness. The  paper by Khanh et al. \cite{BorisKhanhPhat} developed a coderivative-based algorithm of the pure Newton type to solve {\em subgradient inclusions} defined by prox-regular functions with justifying the local superlinear convergence of iterates under the semismooth$^*$ property of the corresponding subgradient mapping. {The question about how to achieve the global convergence of the coderivative-based Newton methods has not been investigated in aforementioned papers and/or any other publication. In particular, the new GDNM algorithms answer in the affirmative the open question formulated in \cite{BorisEbrahim} on whether coderivative-based Newton methods can be globalized via a damping strategy.}\vspace{0.02in}

As mentioned, this paper  addresses not only  the design and justification of new globally convergent algorithms for unconstrained problems of ${\cal C}^{1,1}$ optimization,   but also develops such algorithms for generally {\em constrained} problems of convex composite optimization. For the latter class, we employ the {\em forward-backward envelope} (FBE), the construction that has been recently introduced in variational analysis and optimization, while has been already proved useful in constrained optimization; see, e.g., \cite{stp2} with the references therein. 

A central assumption in our {\em GDNM algorithm} for problems of {\em ${\cal C}^{1,1}$ optimization} is the {\em positive-definiteness} of the generalized Hessian $\partial^2\ph$, which is a direct extension of that for the classical Hessian $\nabla^2\ph$ in the damped Newton method. This assumption alone ensures that GDNM for such problems is {\em well-defined} and converges globally to a {\em tilt-stable minimizer} of \eqref{opproblem} with at least some {\em linear rate}. The $Q$-{\em superlinear} convergence rate of GDNM is guaranteed under the semismooth$^*$ property of the gradient mapping $\nabla\ph$ and some relationship between parameters of the algorithm and the problem data.

The proposed {\em {GRNM algorithm}} to solve problems of ${\cal C}^{1,1}$ optimization does not generally require the positive-definiteness of the generalized Hessian $\partial^2\ph$: we construct it and verify its {\em well-posedness} and {\em global convergence} to stationary points of $\ph$ under merely {\em positive-semidefiniteness} of the generalized Hessian. To establish results on the linear and superlinear {\em convergence rates} of  {GRNM}, the {\em metric regularity} of the gradient mappings is additionally imposed. The latter property has been well understood, characterized, and broadly applied in variational analysis, optimization, and related areas.

Considering further problems of {\em convex composite optimization} in the form
\begin{equation}\label{comp}
\min\;\varphi(x):=f(x)+g(x)\;\text{ subject to }\;x\in\R^n,
\end{equation}
where $f$ is a convex smooth function and $g\colon\R^n\to\oR:=(-\infty,\infty]$ is a lower semicontinuous (l.s.c.) extended-real-valued convex function, we reduce them to ${\cal C}^{1,1}$ optimization by using {\em FBE}. Employing {\em second-order calculus rules} allows us to express the generalized Hessian of FBE via the problem data and relate the metric regularity and tilt stability properties of $\ph$ from \eqref{comp} to the corresponding ones for FBE. In this way, we establish constructive results on well-posedness, global convergence, and convergence rates for both GDNM and  {GRNM} algorithms to solve \eqref{comp} {\em with} and {\em without} the {\em strong convexity} assumption on $f$ in a highly important case of quadratic functions in \eqref{comp}; see below.

The results established by using both GDNM and GRNM for problems of
convex composite optimization with quadratic functions $f$ in \eqref{comp} are employed to solve a basis class of {\em Lasso problems}, which can be written in this form. Such problems were introduced by Tibshirani \cite{Tibshirani} motivated by applications to statistics, and since that they have been largely investigated
and applied to practical models in machine learning, image processing, etc. Computing all the ingredients of both algorithms in Lasso terms, we conduct MATLAB numerical experiments by using random data sets. To compare the performance of GDNM and  {GRNM} with other quite popular and efficient algorithms of nonsmooth optimization, we conduct parallel numerical experiments with the same Lasso data for
the recent second-order {\em Semismooth Newton Augmented Lagrangian
Methods} (SSNAL) developed in \cite{lsk} and the two well-recognized
first-order methods: the {\em Alternating Direction Method of
Multipliers} (ADMM) taken from \cite{BPCPE} and the {\em Fast
Iterative Shrinkage-Thresholding Algorithm} (FISTA) developed in
\cite{BeckTebou}.  In addition, we also provide numerical experiments to solve {\em convex quadratic programming} problems with {\em box constraints}, which arise in many applications as well as subproblems of more complex optimization problems \cite{nw}. The conducted numerical experiments for this part are compared with the {\em trust region reflexive algorithm} \cite{BeckAl}.\vspace*{0.03in}

The subsequent parts of the paper are organized as follows. In Section~\ref{sec:pre}, we briefly overview the tools of variational analysis and generalized differentiation used in our algorithmic developments. Section~\ref{sec:dampedC11} describes and justifies the coderivative-based GDNM to solve problems of ${\cal C}^{1,1}$ optimization. In Section~\ref{sec:modifiedC11}, we design the coderivative-based {GRNM} for the same class of problems with deriving well-posedness and global convergence results. Section~\ref{sec:dampnon} develops both GDNM and {GRNM} to solve problems of convex composite optimization. Section~\ref{Lassosec} is devoted to numerical experiments for our methods and their comparison with the standard semismooth Newton method in $\mathcal{C}^{1,1}$ optimization. Then we conduct numerical experiments to employ GDNM and GRNM for solving a basic class of Lasso problems and compare the achieved numerical results with those obtained by using SSNAL, ADMM, and FISTA. This section also contains applications and numerical experiments to solve box constrained problems of quadratic programming. The concluding Section~\ref{sec:conclusion} lists the main achievements of the paper and discusses some topics of our future research. For the reader's convenience, we place several technical lemmas in the Appendix.
\vspace*{-0.15in}

\section{Preliminaries from Variational Analysis}\label{sec:pre}\vspace*{-0.05in}

 For the reader's convenience,   this section presents some preliminaries from variational analysis and generalized differentiation that are broadly employed  in what follows. More details can be found in the monographs \cite{Mordukhovich06,Mor18,Rockafellar98} from which we borrow the standard notation used below. Recall that $\N:=\{1,2,\ldots\}$.

Given a set-valued mapping (multifunction) $F\colon\R^n\tto\R^m$ between finite-dimensional spaces, its (sequential Painlev\'e-Kuratowski) {\em outer limit} at $\ox$ is defined by
\begin{equation*}
\Limsup_{x\to\ox} F(x):
=\big\{y\in\R^n\;\big|\;\exists\,\mbox{ sequences }\;x_k\to\bar x,\;y_k\rightarrow y\;\mbox{ with }\;y_k\in F(x_k),\;\;k\in\N\big\}.
\end{equation*}
Using the notation $z\st{\O}{\to}\oz$ meaning that $z\to\oz$ with $z\in\O$ for a given nonempty set $\O\subset\R^s$, the (Fr\'echet) {\em regular normal cone} to $\Omega$ at $\bar{z}\in\Omega$ is
\begin{equation*}
\widehat{N}_\Omega(\bar{z}):=\Big\{v\in\R^s\;\Big|\;\limsup_{z\overset{\Omega}{\rightarrow}\bar{z}}\frac{\langle v,z-\bar{z}\rangle}{\|z-\bar{z}\|}\le 0\Big\},
\end{equation*}
while the (Mordukhovich) {\em limiting normal cone} to $\Omega$ at $\bar{z}\in\Omega$ is defined by
\begin{equation}\label{lnc}
N_\Omega(\bar{z}):=  \underset{z\overset{\Omega}{\to} \bar{z}} {\Limsup}\;\widehat{N}_\Omega (z) =\big\{v\in\R^s\;\big|\;\exists\,z_k\st{\O}{\to}\bar{z},\;v_k\to v\;\text{ as }\;k\to\infty\;\text{ with }\;v_k\in\widehat{N}_\Omega(z_k)\big\}.
\end{equation}
The corresponding limiting {\em coderivative} of $F\colon\R^n\tto\R^m$ at $(\ox,\oy)\in\gph F$ is defined    by
\begin{equation}\label{lim-cod}
D^*F(\ox,\oy)(v):=\big\{u\in\R^n\;\big|\;(u,-v)\in N_{{\rm gph}\,F}(\ox,\oy)\big\},\quad v\in\R^m,
\end{equation}
where $\gph F:=\{(x,y)\in\R^n\times\R^m\;|\;y\in F(x)\}$, and where $\oy$ is omitted in the coderivative notation if $F(\bar{x})$ is a singleton. If $F\colon\R^n\to\R^m$ is a single-valued mapping which is ${\cal C}^1$-smooth around $\ox$, then
\begin{equation*}
D^*F(\bar{x})(v)=\big\{\nabla F(\bar{x})^*v\big\}\;\mbox{ for all }\;v\in\R^m
\end{equation*}
via the adjoint/transpose Jacobian matrix $\nabla F(\ox)^*$. The defined coderivative of general multifunctions satisfies comprehensive {\em calculus rules} based on {\em variational/extremal principles} of variational analysis. Among the most impressive and useful advantages of the coderivative \eqref{lim-cod} are complete characterizations in its terms the fundamental well-posedness properties (metric regularity, linear openness, and Lipschitzian behavior) of general multifunctions that were developed in \cite{Mordu93} and were labeled in \cite{Rockafellar98} as the {\em Mordukhovich criteria}. In this paper we employ these characterizations for the property of {\em metric regularity} of $F\colon\R^n\tto\R^m$ around $(\ox,\oy)\in\gph F$ meaning that there exist a number $\mu>0$ and neighborhoods $U$ of $\bar{x}$ and $V$ of $\bar{y}$ such that
\begin{equation}\label{metreg}
{\rm dist}\big(x;F^{-1}(y)\big)\le\mu\,{\rm dist}\big(y;F(x)\big)\;\text{ for all }\;(x,y)\in U\times V,
\end{equation}
where $F^{-1}(y):=\{x\in\R^n\;|\;y\in F(x)\}$, and where `dist' stands for the distance between a point and a set. If in addition $F^{-1}$ has a single-valued localization around $(\bar{y},\bar{x})$, i.e., there exist neighborhoods $U$ of $\ox$ and $V$ of $\oy$ together with a single-valued mapping $\vt\colon V\to U$ such that $\gph F^{-1}\cap(V\times U) =\gph\vartheta$, then $F$ is {\it strongly metrically regular} around $(\bar{x},\bar{y})$ with modulus $\mu>0$. The aforementioned coderivative characterization from \cite[Theorem~3.6]{Mordu93} tells us that, whenever $F$ is closed-graph around $(\ox,\oy)\in\gph F$, its metric regularity around this point is equivalent to the implication
\begin{equation}\label{Morcri}
\big[v\in \R^m,\;0\in D^*F(\bar{x},\bar{y})(v)\big]\Longrightarrow v=0.
\end{equation}
Moreover, the {\em exact regularity bound} of $F$ at $(\bar{x},\bar{y})$, i.e., the infimum  of all $\mu>0$ such that \eqref{metreg} holds for some neighborhoods $U$ and $V$, is calculated by
\begin{equation}\label{Morcri2}
\text{\rm reg}\;F(\bar{x},\bar{y})=\|D^*F(\bar{x},\bar{y})^{-1}\|=\|D^*F^{-1}(\bar{y},\bar{x})\|
\end{equation}
via the norm of the coderivatives of $F$ and $F^{-1}$ as positive homogeneous multifunctions; see \cite{Mordukhovich06,Mor18,Rockafellar98} for more discussions, different proofs, and infinite-dimensional extensions.

 Another notion in variational analysis used in what follows concerns a strong version of local monotonicity for set-valued mappings. We say that $T\colon\R^n\rightrightarrows \R^n$ is \textit{strongly locally monotone} with modulus $\kappa>0$ around $(\ox,\oy)\in\gph T$ if there exist neighborhoods $U$ of $\ox$ and $V$ of $\oy$ such that
$$
\langle x-u,v-w\rangle\ge\kappa\,\|x-u\|^2\quad\text{for all }\;(x,v),(u,w)\in\gph T\cap(U\times V).
$$
If in addition $\gph T\cap(U\times V) =\gph S\cap(U\times V)$ for any monotone operator $S:\R^n\rightrightarrows \R^n$ satisfying $\gph T\cap(U\times V)\subset\gph S$, then $T$ is {\em strongly locally maximal monotone} with modulus $\kappa>0$ around $(\ox,\oy)$. The reader is referred to \cite{MorduNghia1} and \cite[Section~5.2]{Mor18} for coderivative characterizations of the latter property, which is significantly more relaxed than the strong metric regularity of $T$ around $(\ox,\oy)$.\vspace*{0.03in}

Next we consider an extended-real-valued function $\ph\colon\R^n\to\oR$ with the domain and epigraph
\begin{equation*}
\dom\ph:=\big\{x\in\R^n\;\big|\;\ph(x)<\infty\big\},\quad\epi\ph:=\big\{(x,\al)\in\R^{n+1}\;\big|\;\al\ge\ph(x)\big\}.
\end{equation*}
 {Recall that an l.s.c. proper function $\varphi:\R^n\to\overline{\R}$ is {\em strongly convex} on a convex set $\Omega\subset\R^n$ with modulus $\kappa>0$ if the quadratically shifted function $\varphi-(\kappa/2)\|\cdot\|^2$ is convex on $\Omega$.}

The (limiting) {\em subdifferential} of $\ph$ at $\ox\in\dom\ph$ is defined geometrically by
\begin{equation}\label{lim-sub}
\partial\varphi(\ox):=\big\{v\in\R^n\;\big|\;(v,-1)\in N_{{\rm\small epi}\,\varphi}\big(\bar{x},\varphi(\bar{x})\big)\big\}
\end{equation}
via the limiting normal cone \eqref{lnc}, while admitting various analytic representations and satisfying comprehensive calculus rules that can be found in \cite{Mordukhovich06,Mor18,Rockafellar98}. Observe the useful scalarization formula
\begin{equation}\label{scal}
D^*F(\bar{x})(v)=\partial\langle v,F\rangle(\bar{x})\;\mbox{ for all }\;v\in\R^m 
\end{equation}
connecting the coderivative \eqref{lim-cod} of a locally Lipschitzian mapping $F\colon\R^n\to\R^m$ and the subdifferential \eqref{lim-sub} of the function $x\mapsto\langle v,F\rangle(x)$ whenever $v\in\R^m$.

Following \cite{m92}, we define the {\em second-order subdifferential}, or {\em generalized Hessian}, $\partial^2\ph(\ox,\ov)\colon\R^n\tto\R^n$ of $\ph\colon\R^n\to\oR$ at $\ox\in\dom\ph$ for $\ov\in\partial\ph(\ox)$ as the coderivative of the subgradient mapping
\begin{equation}\label{2nd}
\partial^2\ph(\ox,\ov)(u):=\big(D^*\partial\ph\big)(\ox,\oy)(u)\;\mbox{ for all }\;u\in\R^n.
\end{equation}
If $\ph$ is ${\cal C}^2$-smooth around $\ox$, then we have
\begin{equation}\label{C2_Case}
\partial^2\ph(\ox)(u)=\big\{\nabla^2\ph(\ox)u\big\}\;\mbox{ for all }\;u\in\R^n,
\end{equation}
while for $\ph$ of class ${\cal C}^{1,1}$ around $\ox$, we get by the scalarization formula \eqref{scal} that
\begin{equation}\label{scal1}
\partial^2\ph(\ox)(u)=\partial\big\la u,\nabla\ph\big\ra(\ox)\;\mbox{ for all }\;u\in\R^n.
\end{equation}
As follows from \eqref{scal1}, calculus rules and computations of the second-order subdifferential for ${\cal C}^{1,1}$ functions reduce in fact to those for the first-order construction \eqref{lim-sub}. We also have well-developed second-order calculus rules for \eqref{2nd} for rather general classes of extended-real-valued functions; see, e.g., \cite{Mordukhovich06,Mor18,mr} with many additional references. Furthermore, the second-order subdifferential has been computed and analyzed in terms of the given data for broad classes of structural functional systems appearing in numerous aspects of variational analysis, optimization, stability, and optimal control among other areas, with subsequent applications to optimality conditions, sensitivity analysis, numerical algorithms, stochastic programming, electricity markets, etc. The reader can find more information in, e.g., \cite{ChieuLee17,chhm,dsy,dr,hmn,hr,m92,Mordukhovich06,Mor18,BorisOutrata,mos,mr,mrs,roc,yy} along with  other publications on such developments and related topics of second-order variational analysis. Some new results in this direction are presented in what follows. \vspace*{0.03in}

In this paper, we use the fundamental notion of {\em tilt-stable local minimizers} and its second-order characterizations for the justification of the proposed Newton-type algorithms.

\begin{Definition}[\bf tilt-stable local minimizers]\label{def:tilt} Given $\varphi\colon\R^n\to\oR$, a point $\ox\in\dom\varphi$ is a {\sc tilt-stable local minimizer} of $\varphi$ if there exists a number $\gamma>0$ such that the mapping
\begin{equation*}
M_\gamma\colon v\mapsto{\rm argmin}\big\{\varphi(x)-\langle v,x\rangle\;\big|\;x\in\B_\gamma(\ox)\big\}
\end{equation*}
is single-valued and Lipschitz continuous on some neighborhood of $0\in\R^n$ with $M_\gamma(0)=\{\ox\}$. By a {\sc modulus} of tilt stability of $\ph$ at $\ox$ we understand a Lipschitz constant of $M_\gamma$ around the origin.
\end{Definition}

This notion was introduced by Poliquin and Rockafellar in \cite{Poli} and characterized there via $\partial^2\ph$ for a broad class of prox-regular functions $\ph\colon\R^n\to\oR$ that are overwhelmingly involved in second-order variational analysis. More recently, developing second-order subdifferential calculus and second-order growth conditions made it possible to establish complete characterizations of tilt-stable local minimizers for various classes of problems in constrained optimization including nonlinear programming, extended nonlinear programming, composite optimization, minimax problems, second-order cone programming, semidefinite programming, etc.; see, e.g., \cite{ChieuNghia,dl,dmn,gm,Mor18,MorduNghia,mr} among other publications on tilt stability in optimization.\vspace*{0.03in}

Finally in this section, we recall  for completeness  the notions of convergence rates used for our algorithms.

\begin{Definition}[\bf rates of convergence]\label{def-rates} Let $\{x^k\}\subset\R^n$ be a sequence of vectors converging to $\ox$ as $k\to\infty$ with $\bar{x}\ne x^k$ for all $k\in\N$. The convergence rate is said to be:
\begin{itemize}
    
\item[\bf(i)] \textsc{R-linear} if we have
$$
0<\limsup_{k\to\infty}\left(\|x^k-\bar{x}\|\right)^{1/k}<1,
$$
i.e., there exist $\mu\in(0,1)$, $c>0$, and $k_0\in\N$ such that
$$
\|x^{k}-\bar{x}\|\le c\mu^k\quad\text{for all }\;k\ge k_0.
$$

\item[\bf(ii)] \textsc{Q-linear} if we have
$$
\limsup_{k\to\infty}\frac{\|x^{k+1}-\bar{x}\|}{\|x^k-\bar{x}\|}<1,
$$
i.e., there exist $\mu\in(0,1)$ and $k_0\in\N$ such that
$$
\|x^{k+1}-\bar{x}\|\le\mu\|x^k-\bar{x}\|\quad\text{for all }\;k\ge k_0.
$$

\item[\bf(iii)] \textsc{Q-superlinear} if we have
$$
\lim_{k\to\infty}\frac{\|x^{k+1}-\bar{x}\|}{\|x^k-\bar{x}\|}=0.
$$
\end{itemize}
\end{Definition}\vspace*{-0.15in}

\section{Coderivative-Based Damped Newton Method in $\mathcal{C}^{1,1}$ Optimization}\label{sec:dampedC11}

In this section, we concentrate on the unconstrained optimization problem \eqref{opproblem}, where the cost function $\varphi\colon\R^n\to\R$ is of class $\mathcal{C}^{1,1}$.
This kind of problem plays a crucial role not only in numerical optimization but also in applied areas including, e.g., machine learning. In particular, $\mathcal{C}^{1,1}$ optimization problems also arise frequently as {\it  subproblems} in  {\it augmented Lagrangian methods} \cite{hangborisebrahim,Roc74,roc}. Furthermore, L2-loss support vector regression problems are important classes of problems that are $\mathcal{C}^{1,1}$ unconstrained optimization problem \cite{hc12}. More practical examples about $\mathcal{C}^{1,1}$ functions  can be found in the paper \cite{hsn84}.
A coderivative-based generalization of the pure Newton method to solve \eqref{opproblem} {\em locally} was first suggested and investigated in \cite{BorisEbrahim} under the major assumption that a given point $\ox$ is a {\em tilt-stable} local minimizer of \eqref{opproblem}. Then it was extended in \cite{BorisKhanhPhat} to solve directly the gradient system $\nabla\varphi(x) =0$ under certain assumptions on a given solution $\ox$ of the gradient equation ensuring the well-posedness and local superlinear convergence of the algorithm. One of the serious disadvantages of the pure Newton method and its generalizations is that the corresponding sequence of iterates may not converge if the starting point is not sufficiently close to the solution. This motivates us to design and justify a {\em globally} convergent {\em damped Newton} counterpart of the generalized pure Newton algorithms from \cite{BorisKhanhPhat,BorisEbrahim} with {\em backtracking line search} to solve \eqref{opproblem}. Here is the algorithm.

\begin{algorithm}[H]
\caption{{Coderivative-based damped Newton algorithm for $\mathcal{C}^{1,1}$ functions}}\label{LS}
\begin{algorithmic}[1]
\Require {$x^0\in\R^n$,  $\sigma\in\left(0,\frac{1}{2}\right)$,  $\beta\in\left(0,1\right)$}
\For{\texttt{$k=0,1,\ldots$}}
\State\text{If $\nabla\varphi(x^k)=0$, stop; otherwise  go to the next step}
\State \label{alC11} \text{Choose $d^k\in\R^n$ such that
$-\nabla\varphi(x^k)\in\partial^2\ph(x^k)(d^k)$}
\State Set $\tau_k = 1$
\While{$\varphi(x^k+\tau_kd^k)>\varphi(x^k)+\sigma\tau_k\langle d^k,\nabla\varphi(x^k)\rangle$}
\State\text{set $\tau_k:= \beta\tau_k$}
\EndWhile\label{euclidendwhile}
\State \text{Set $x^{k+1}:=x^k+\tau_k d^k$}
\EndFor
\end{algorithmic}
\end{algorithm}

If $\varphi$ is $\mathcal{C}^2$-smooth, Algorithm~\ref{LS} reduces to the standard damped Newton method (as, e.g., in \cite{BeckAl,Boyd}) due to \eqref{C2_Case}. In the general case of $\ph\in
{\cal C}^{1,1}$, it follows from \eqref{lim-cod} that the direction $d^k$ in  {Step 3 of Algorithm \ref{LS}} can be explicitly found from the inclusion
\begin{equation*}
\big(-\nabla\varphi(x^k),-d^k\big)\in N\big((x^k,\nabla\varphi(x^k));\gph\nabla\varphi\big).
\end{equation*}
Note also that, due to the scalarization formula \eqref{scal1}, the Newton equation in Step~3 of Algorithm~\ref{LS} can be equivalently written in the form
\begin{equation}\label{damped-dir}
-\nabla\varphi(x^k)\in\partial\la d^k,\nabla\ph\ra(x^k),
\end{equation}
which merely requires the first-order subdifferential computation.

We start justifying Algorithm~\ref{LS} with the verification of its {\em well-posedness}. The following proposition establishes the existence of {\em descent} Newton directions under the {\em positive-definiteness} of $\partial^2\ph(x)$.

\begin{Proposition}[\bf {existence of Newton directions and descent property}]\label{descent} Let $\varphi\colon\R^n\to\R$ be of class
$\mathcal{C}^{1,1}$ on $\R^n$. Suppose that $\nabla\varphi(x)\ne 0$
and that  {$\partial^2\varphi(x)$ is
positive-definite}, i.e., \begin{equation}\label{pos-def}
\langle v,u\rangle>0\;\mbox{ for all
}\;v\in\partial^2\varphi(x)(u)\;\mbox{ and }\;u\ne 0.
\end{equation}
Then there exists a nonzero direction $d\in\R^n$ such that
\begin{equation}\label{descentforvaphi}
-\nabla\varphi(x)\in\partial^2\varphi(x)(d). \end{equation}
Moreover, every such direction satisfies the inequality $\langle
\nabla\varphi(x),d\rangle<0$. Consequently, for each
$\sigma\in(0,1)$ and $d\in\R^n$ satisfying \eqref{descentforvaphi}
we find $\delta>0$ such that \begin{equation}\label{Armijo}
\varphi(x+\tau
d)\le\varphi(x)+\sigma\tau\langle\nabla\varphi(x),d\rangle\;\mbox{
whenever }\;\tau\in(0,\delta). 
\end{equation} \end{Proposition}
\begin{proof} By the positive-definiteness of $\partial^2\ph(x)$, it
follows from \cite[Theorem~5.16]{Mor18} that $\nabla\varphi$ is
 {strongly locally maximal monotone} around $(x,\nabla\varphi(x))$.
Thus $\nabla\varphi$ is strongly metrically regular around
$(x,\nabla\varphi(x))$ by \cite[Theorem~5.13]{Mor18}. Using
\cite[Corollary~4.2]{BorisKhanhPhat} yields the existence of
$d\in\R^n$ with $-\nabla\varphi(x)\in\partial^2\varphi(x)(d)$. To
verify further that $d\ne 0$, suppose on the contrary that $d=0$.
Since $\nabla\varphi$ is locally Lipschitz around $x$, it follows
from \cite[Theorem~1.44]{Mordukhovich06} that $$
-\nabla\varphi(x)\in\partial^2\varphi(x)(0)=\big(D^*\nabla\varphi\big)(x)(0)=\{0\},
$$ which contradicts the assumption that $\nabla\varphi(x)\ne 0$.
Employing again the positive-definiteness of $\partial^2\varphi$
tells us that $\langle\nabla\varphi(x),d\rangle<0$. Using
\cite[Lemmas~2.18 and 2.19]{Solo14}, we arrive at \eqref{Armijo} and thus complete the proof. \end{proof}

The next theorem establishes the {\em global linear convergence} of Algorithm~\ref{LS} to a {\em tilt-stable minimizer} of \eqref{opproblem} under the positive-definiteness assumption on the generalized Hessian $\partial^2\ph$.

\begin{Theorem}[\bf global linear convergence of the coderivative-based damped Newton algorithm for $\mathcal{C}^{1,1}$ functions]\label{globalconver} Let $\varphi\colon\R^n\to\R$ be of class $\mathcal{C}^{1,1}$, and let $x^0\in\R^n$ be an arbitrary point such that the generalized Hessian $\partial^2\varphi(x)$ is positive-definite for all $x\in\Omega$, where
\begin{equation}\label{level}
\Omega:=\big\{x\in\R^n\;\big|\;\varphi(x)\le\varphi(x^0)\big\}.
\end{equation}
Then Algorithm~{\rm\ref{LS}} either stops after finitely many iterations, or produces a sequence $\{x^k\}\subset\Omega$ such that $\{\varphi(x^k)\}$ is monotonically decreasing. Moreover, if the iterative sequence $\{x^k\}$ has  {an accumulation point} $\bar{x}$ $($in particular, when the level set $\Omega$ from \eqref{level} is bounded$)$, then $\{x^k\}$ converges to $\bar{x}$, which is a tilt-stable local minimizer of $\varphi$. In this case, we have:

{\bf(i)} The convergence rate of $\{\varphi(x^k)\}$ is at least Q-linear.

{\bf(ii)} The convergence rates of $\{x^k\} $ and  $\{\|\nabla\varphi(x^k)\|\}$ are at least R-linear.
\end{Theorem}
\begin{proof} Proposition~\ref{descent} easily ensures by induction that Algorithm~\ref{LS} either stops after finitely many iterations, or produces a sequence $\{x^k\}\subset\Omega$ such that $\varphi(x^{k+1})<\varphi(x^k)$ for all $k\in\N$. Suppose next that $\{x^k\}$ has  {an accumulation point} $\bar{x}$. Since the set $\Omega$ is closed, we get $\bar{x}\in\Omega$, and hence have that $\partial^2\varphi(\bar{x})$ is positive-definite. Then \cite[Proposition~4.6]{ChieuLee17} gives us positive numbers $\kappa$ and $\delta$ such that
\begin{equation}\label{uniformPD}
\langle z,w\rangle\ge\kappa\|w\|^2\quad\text{for all }\;z\in\partial^2\varphi(x)(w),\;x\in\mathbb{B}_\delta(\bar{x}),\;\mbox{ and }\;w\in\R^n.
\end{equation}
Since $\ph$ is of class ${\cal C}^{1,1}$ around $\bar{x}$, we get without loss of generality that $\nabla\varphi$ is Lipschitz continuous on $\mathbb{B}_\delta(\bar{x})$ with some constant $\ell>0$. By  \cite[Theorem~1.44]{Mordukhovich06} we have
\begin{equation}\label{uniformLip}
\|z\|\leq\ell\|w\|\quad\text{for all }\;z\in\partial^2\varphi(x)(w),\;x\in\mathbb{B}_\delta(\bar{x}),\;\mbox{ and }\;w\in\R^n.
\end{equation}
The rest of the proof is split into the following four claims.\\[1ex]
 {{\bf Claim~1:} {\em For any subsequence $\{x^{k_j}\}$ of $\{x^k\}$ such that $x^{k_j}\to\bar{x}$ as $j \to\infty$, the corresponding sequence $\{\tau_{k_j}\}$ in Algorithm~{\rm\ref{LS}} is bounded from below by some $\gamma>0$, the corresponding sequence $\{d^{k_j}\}$ is bounded, and we have}
\begin{equation}\label{dk}
\langle-\nabla \varphi(x^{k_j}),d^{k_j}\rangle\ge\kappa \|d^{k_j}\|^2,
\end{equation}
\begin{equation}\label{dkLips}
\|\nabla \varphi(x^{k_j})\| \leq \ell \|d^{k_j}\|,
\end{equation}
\begin{equation}\label{xk+1}
\varphi(x^{k_j})-\varphi(x^{k_j+1}) \ge\sigma\gamma\kappa\|d^{k_j}\|^2,
\end{equation}
for all large $j\in \N$. Since {$x^{k_j}\to \ox$} as $j\to \infty$ and $-\nabla\varphi(x^{k_j})\in\partial^2\varphi(x^{k_j})(d^{k_j})$ for all $j\in\N$, we obtain \eqref{dk} and \eqref{dkLips} from \eqref{uniformPD} and \eqref{uniformLip}, respectively. The Cauchy-Schwarz inequality yields $\|\nabla\varphi(x^{k_j})\|\ge\kappa\|d^{k_j}\|$ for such $j$. Since $x^{k_j}\to \ox$ as $j \to \infty$, the latter estimate verifies the boundedness of the sequence of directions $\{d^{k_j}\}$. It remains to show that  $\{\tau_{k_j}\}_{j\in\N}$ is bounded from below by a positive number. Indeed, supposing on the contrary that the opposite holds and combining this with $\tau_k\ge 0$ give us a subsequence of $\{\tau_{k_j}\}$ that converges to $0$. Assume without loss of generality that $\tau_{k_j}\to 0$ as $j\to\infty$.
 Thus $x^{k_j}+\tau_{k_j}d^{k_j}\to\bar{x}$ as $j\to\infty$, and hence $
x^{k_j}+\tau_{k_j}d^{k_j}\in\text{\rm int}\,\mathbb{B}_\delta(\bar{x})$ whenever $j$ is sufficiently large. Applying Lemma~\ref{estimate1} from the Appendix, we have 
$$
\beta^{-1}\tau_{k_j} > \frac{2(\sigma -1)\langle \nabla \varphi(x^{k_j}),d^{k_j}\rangle}{\ell\|d^{k_j}\|^2} \geq \frac{2(1-\sigma)\kappa}{\ell},
$$
where the second inequality follows from \eqref{dk}.  
Letting $j\to\infty$ gives us $\sigma\ge 1$, a contradiction due to the choice of $\sigma$. Hence there exists $\gamma>0$ such that $\tau_{k_j}\ge\gamma$ for all $j \in \N$. Moreover, using the estimate in \eqref{dk} allows us to find $j_0\in\N$ such that
\begin{equation}\label{ineq}
\varphi(x^{k_j})-\varphi(x^{k_j+1})\ge\sigma\tau_{k_j}\langle-\nabla\varphi(x^{k_j}),d^{k_j}\rangle\ge\sigma\gamma\kappa\|d^{k_j}\|^2 \quad \text{for all }\;j\ge j_0,
\end{equation}
which therefore justifies Claim 1.}\\[1ex]
{\bf Claim~2:} {\em $\bar{x}$ is a tilt-stable local minimizer of $\varphi$.} To verify this, we only need to show that $\bar{x}$ is a stationary point of $\varphi$, by taking into account the positive-definiteness of $\partial^2\varphi(\bar{x})$ and the second-order characterization of tilt-stability from \cite[Theorem~1.3]{Poli}.  Since $\bar{x}$ is {an accumulation point} of $\{x^k\}$,  there exists a subsequence $\{x^{k_j}\}_{j\in\N}$ of $\{x^k\}$ such that $x^{k_j}\to \bar{x}$ as $j\to\infty$. Due to Claim~1, we find $\gamma>0$ such that \eqref{xk+1} is satisfied. Since the sequence $\{\varphi(x^k)\}$ is nonincreasing and since $\varphi(\bar{x})$ is {an accumulation point} of $\{\varphi(x^k)\}$, the sequence $\{\varphi(x^k)\}$ must converge to $\varphi(\bar{x})$ as $k\to\infty$. Letting $j\to\infty$ in the inequality \eqref{xk+1}, we have $\|d^{k_j}\| \to 0$ as $j\to\infty$.  {Passing to the limit as $j\to\infty$ in the inequality $\|\nabla\varphi(x^{k_j})\|\le\ell\|d^{k_j}\|\quad\text{for all large }\;j$ from \eqref{dkLips} tells us that $\nabla\varphi(\bar{x})=0$, which readily justifies Claim~2.} \\[1ex]
{\bf Claim~3:} {\em  The iterative sequence $\{x^k\}$ is convergent.} To verify this, we use Ostrowski's condition from \cite[Proposition~8.3.10]{JPang}. First we show that no other  {accumulation point} of $\{x^k\}$ exists in $\mathbb{B}_\delta(\bar{x})$. Assuming the contrary, find $\Tilde x\in\mathbb{B}_\delta(\bar{x})$ such that $\Tilde x\ne\bar{x}$ and that $\Tilde x$ is  {an accumulation point} of $\{x^k\}$. Arguing similarly to Claim~2 tells us that $\Tilde x$ is a tilt-stable local minimizer of $\varphi$, which contradicts the strong convexity of $\varphi$ on $\mathbb{B}_\delta(\bar{x})$. Supposing next that $\{x^{k_j}\}$ is an arbitrary subsequence of $\{x^k\}$ with $x^{k_j}\to\bar{x}$ as $j\to\infty$, we need to check that
\begin{equation}\label{Ostrowski}
\lim_{j\to\infty}\|x^{k_j+1}- x^{k_j}\| =0.
\end{equation}
Indeed, Claim~1 gives us $\gamma>0$ such that \eqref{xk+1} holds, which implies in turn that
$$
\|x^{k_j+1}- x^{k_j}\|^2= \tau_{k_j}^2 \|d^{k_j}\|^2 \le\|d^{k_j}\|^2 \le\frac{1}{\sigma\gamma\kappa}\left(\varphi(x^{k_j})-\varphi(x^{k_j+1}) \right)\to 0
$$
and hence verifies \eqref{Ostrowski}. Finally, it follows from \cite[Proposition~8.3.10]{JPang} that the sequence $\{x^k\}$ converges to $\bar{x}$ as $k\to\infty$, which therefore completes the proof of Claim~3. \\[1ex]
{\bf Claim~4:} {\em The convergence rate of $\{\varphi(x^k)\}$ is at least Q-linear, while the convergence rates of $\{x^k\}$ and  $\{\|\nabla\varphi(x^k)\| \}$ are at least R-linear.}  Indeed, the strong convexity of $\varphi$ on $\mathbb{B}_\delta(\bar{x})$ shows that
\begin{equation}\label{strongineq1}
\langle\nabla\varphi(x)-\nabla\varphi(u),x-u\rangle\ge\kappa\|x-u\|^2\;\mbox{ for all }\;x,u\in\mathbb{B}_\delta(\bar{x}).
\end{equation}
Since $x^k\to\bar{x}$ as $k\to\infty$, we have that $x^k\in U$ for all $k\in\N$ sufficiently large. Substituting $x:=x^k$ and $u:=\bar{x}$ into \eqref{strongineq1} and then using the Cauchy-Schwarz inequality together with the   {stationarity condition} $\nabla\varphi(\bar{x})=0$ give us the lower estimate
\begin{equation}\label{strongineq}
\|\nabla\varphi(x^k)\|\ge\kappa\|x^k-\bar{x}\|
\end{equation}
for large $k$. The local Lipschitz continuity of $\nabla\varphi$ around $\bar{x}$ and the result of \cite[Lemma~A.11]{Solo14} ensure the existence of $\ell>0$ such that
\begin{equation}\label{Lipschitzineq}
\varphi(x^k)-\varphi(\bar{x})=|\varphi(x^k)-\varphi(\bar{x})-\langle\nabla\varphi(\bar{x}),x^k-\bar{x}\rangle|\le\frac{\ell}{2}\|x^k-\bar{x}\|^2
\quad\text{for large }\;k.
\end{equation}
Furthermore, estimate \eqref{uniformLip} together with the inclusion $-\nabla\varphi(x^{k})\in\partial^2\varphi(x^{k})(d^{k})$ implies that
\begin{equation}\label{bounded1Lip1}
\|\nabla\varphi(x^{k})\|\le\ell\|d^{k}\|\quad\text{for large }\;k.
\end{equation}
Claim~1 tells us that the sequence $\{\tau_k\}$ is bounded from below by some constant $\gamma>0$ such that
$$
\varphi(x^{k})-\varphi(x^{k+1})\ge\sigma\gamma\kappa\|d^{k}\|^2\quad\text{for large}\;k.
$$
Combining the above inequality with \eqref{bounded1Lip1} yields the estimates
\begin{equation}\label{ineq2}
\varphi(x^{k})-\varphi(x^{k+1})\ge\sigma\gamma\kappa\ell^{-2}\|\nabla\varphi(x^{k})\|^2
\end{equation}
if $k$ is large. Using \eqref{strongineq}, \eqref{Lipschitzineq}, \eqref{ineq2} and then applying Lemma~\ref{QRlinear} from the Appendix with 
\begin{equation*}
\alpha_k:=\varphi(x^k)-\varphi(\bar{x}),\;\beta_k:=\|\nabla\varphi(x^k)\|,\;\gamma_k:=\|x^k-\bar{x}\|
\end{equation*}
$c_1:=\sigma\gamma\kappa \ell^{-2}$, $c_2:= \kappa$, and $c_3=\ell/2$ verify Claim~4 and thus completes the proof of the theorem.
\end{proof}
\begin{Remark}[\bf on proof of Theorem~\ref{globalconver}]\rm  {Following the suggestion of the referee, we present an alternative proof of Claim~2 in Theorem~\ref{globalconver} by assuming the contrary and using the result of \textit{gradient related property} introduced in \cite{Bert}. Indeed, suppose that $\nabla\varphi(\ox)\ne 0$, we check that the sequence $\{d^k\}$ is \textit{gradient related} to $\{x^k\}$ in the sense of \cite{Bert}, i.e., for any subsequence $\{x^{k_j}\}$ that converges to $\ox$ as $j\to\infty$, the corresponding subsequence $\{d^{k_j}\}$ is bounded and satisfies the inequality
\begin{equation}\label{gradientrelated}
\limsup_{j\to\infty}\langle\nabla\varphi(x^{k_j}),d^{k_j}\rangle <0. 
\end{equation}
The boundedness of $\{d^{k_j}\}$ is clarified in Claim~1. Combining \eqref{dk} and \eqref{dkLips}, we deduce that
$$
\langle\nabla\varphi(x^{k_j}),d^{k_j}\rangle\le-\kappa\|d^{k_j}\|^2 \le-\kappa\ell^{-2}\|\nabla\varphi(x^{k_j})\|^2, \quad\text{for large }\;j\in\N.
$$
Taking the the upper limit above and using $\nabla\varphi(\ox)\ne 0$ justifies \eqref{gradientrelated}. Therefore, it follows from \cite[Proposition~1.2.1]{Bert} that $\nabla\varphi(\ox)=0$. This is a contradiction, which verifies Claim~2.} 
\end{Remark}
	
Our next goal is to establish the {\em Q-superlinear convergence} of Algorithm~\ref{LS}. First recall additional notions of variational analysis and generalized differentiation needed for these developments. A highly recognized concept used for Newton-type methods addressing single-valued Lipschitz continuous mappings is semismoothness. A mapping $f\colon\R^n\to\R^m$ is {\em semismooth} at $\bar{x}$ if it is locally Lipschitzian around this point and the limit
\begin{equation}\label{semi-sm}
\lim_{A\in\text{\rm co}\overline{\nabla}f(\bar{x}+tu')\atop u'\to u,t\downarrow 0}Au'
\end{equation}
exists for all $u\in\R^n$, where `co' stands for the convex hull of a set, and where $\overline{\nabla}f$ is defined by
$$
\overline{\nabla}f(x):=\big\{A\in\R^{m\times n}\;\big|\;\exists\ x_k\overset{\Omega_{f}}\to x\;\text{ such that }\;\nabla f(x_k)\to A\big\},\quad x\in\R^n
$$
with $\Omega_{f}:=\{x\in\R^n\;|\;f\;\text{ is differentiable at }\;x\}$; see \cite{JPang,Solo14,Klatte,qs} for further discussions. Note that   any semismooth mapping automatically admits the classical directional derivative at the reference point.\vspace*{0.02in}

The next proposition about the acceptance of the unit stepsize under semismoothness follows from a close look at the proof of
\cite[Theorem~3.3]{F1996}.

\begin{Proposition}[\bf acceptance of unit stepsize under semismoothness]\label{acceptancesm} Suppose that a function $\varphi:\R^n\to\R$ is $\mathcal{C}^{1}$-smooth around $\ox$ with $\nabla \varphi(\ox)=0$, and that $\nabla\varphi$ is semismooth at this point. Let a sequence $\{x^k\}$ converge to $\bar{x}$ with $x^k\ne\bar{x}$ as $k\in\N$, and let a sequence $\{d^k\}$ satisfy the condition 
\begin{equation}\label{dksuperlinear}
\|x^k+d^k-\bar{x}\|=o(\|x^k-\bar{x}\|).
\end{equation} 
Suppose further that there exists $\kappa>0$ such that
$\langle\nabla\varphi(x^k),d^k\rangle\le\kappa^{-1}\|d^k\|^2$ whenever $k\in\N$ is sufficiently large. Then for any $\sigma\in (0,1/2)$ we have the estimate
\begin{equation}\label{backtrackinghold}
\varphi(x^k+d^k)\le\varphi(x^k)+\sigma\langle\nabla\varphi(x^k),d^k\rangle\;\mbox{ for all large }\;k.
\end{equation} 
\end{Proposition}

Quite recently \cite{Helmut}, the concept of semismoothness has been improved and extended to set-valued mappings. To formulate the latter notion, recall first the construction of the {\em directional limiting normal cone} to a set $\Omega\subset\R^s$ at $\oz\in\O$ in the direction $d\in\R^s$ introduced in \cite{gin-mor} by
\begin{equation}\label{dir-nc}
N_\Omega(\oz;d):=\big\{v\in\R^s\;\big|\;\exists\,t_k\dn 0,\;d_k\to d,\;v_k\to v\;\mbox{ with }\;v_k\in\widehat{N}_\Omega(\oz+t_k d_k)\big\}.
\end{equation}
It is obvious that \eqref{dir-nc} agrees with the limiting normal cone \eqref{lnc} for $d=0$. The {\em directional limiting coderivative} of $F\colon\R^n\tto\R^m$ at $(\ox,\oy)\in\gph F$ in the direction $(u,v)\in\R^n\times\R^m$ is defined in \cite{g} by
\begin{equation}\label{dir-cod}
D^*F\big((\ox,\oy);(u,v)\big)(v^*):=\big\{u^*\in\R^n\;\big|\;(u^*,-v^*)\in N_{\text{gph}\,F}\big((\ox,\oy);(u,v)\big)\big\}\;\mbox{ for all }\;v^*\in\R^m.
\end{equation}
Using \eqref{dir-cod}, we come to the aforementioned property of set-valued mappings introduced in \cite{Helmut}.

\begin{Definition}[\bf semismooth$^*$ property of set-valued mappings]\label{semi*} A mapping $F\colon\R^n\tto\R^m$ is {\sc semismooth$^*$} at $(\bar{x},\bar{y})\in\gph F$ if whenever $(u,v)\in\R^n\times\R^m$ we have
\begin{equation*}
\langle u^*,u\rangle=\langle v^*,v\rangle\;\mbox{ for all }\;(v^*,u^*)\in\gph D^*F\big((\ox,\oy);(u,v)\big).
\end{equation*}
\end{Definition}
Among various properties of semismooth$^*$ mappings obtained in \cite{Helmut}, recall that this property holds if the graph of $F\colon\R^n\tto\R^m$ is represented as a union of finitely many closed and convex sets, as well as for the normal cone mappings generated by convex polyhedral sets. Note also that the semismooth$^*$ property of single-valued locally Lipschitzian mappings $f\colon\R^n\to\R^m$ around $\ox$ agrees with the semismooth property \eqref{semi-sm} at this point provided that $f$ is {\em directionally differentiable} at $\ox$; see \cite[Corollary~3.8]{Helmut}. Although the standard semismooth property of locally Lipschitzian and directionally differentiable mappings has been conventionally used in the Newton method literature, some important results were obtained without the directional differentiability assumption; see, e.g., Meng et al. \cite{msz}. Such a relaxed semismooth property of single-valued locally Lipschitzian mappings is known as {\em  $G$-semismoothness}. Note that, in contrast to $G$-semismoothness, the semismooth$^*$property is defined for arbitrary set-valued mappings, and it is used for subgradient ones in this paper; see Section \ref{sec:dampnon}. But
even for single-valued Lipschitzian mappings, the semismooth$^*$
definition based on coderivatives may have some advantages in
comparison with the $G$-semismooth one due to comprehensive coderivative calculus rules. More recent results on the semismooth$^*$ property can be found in \cite{fgh}.\vspace*{0.05in} 

The next lemma discusses the acceptance of the unit stepsize for  $\mathcal{C}^{1,1}$ functions with semismooth$^*$ gradients. The obtained estimates are of their own interest, while are instrumental to establish major superlinear convergence results in this and subsequent sections.   

\begin{Lemma}[\bf acceptance of unit stepsize under semismoothness$^*$]\label{Maratos} Let $\varphi\colon\R^n\to\R$ be a ${\cal C}^1$-smooth function around $\bar{x}\in\R^n$ with $\nabla\varphi(\bar{x})=0$. Suppose that $\nabla\varphi$ is locally Lipschitzian around $\ox$ with modulus $\ell>0$, and that $\nabla\varphi$ is semismooth$^*$ at this point. 
Take a sequence $\{x^k\}$ converging to $\bar{x}$ with $x^k\ne\bar{x}$ as $k\in\N$, and let a sequence $\{d^k\}$ satisfy condition \eqref{dksuperlinear}.
Assume also that there exists $\kappa>0$ such that
\begin{equation}\label{growthdk}
\varphi(x^k+d^k)-\varphi(x^k)\le\langle\nabla\varphi(x^k+d^k),d^k\rangle-\frac{1}{2\kappa}\|d^k\|^2
\end{equation}
whenever $k$ is sufficiently large. Then for any $\sigma\in\left(0,1/(2\ell\kappa)\right)$ we have estimate \eqref{backtrackinghold}. 
\end{Lemma}
\begin{proof} Having \eqref{growthdk} with $\sigma\in\left(0,1/(2\ell\kappa)\right)$ and using \eqref{dksuperlinear} together with \cite[Lemma~7.5.7]{JPang}, we get
\begin{equation}\label{superlineark+1}
\lim_{k\to\infty}\|x^k-\bar{x}\|/\|d^k\|=1,
\end{equation}
which also yields the limiting relationship
\begin{equation}\label{supdk}
\|x^k+d^k-\bar{x}\|=o(\|d^k\|)\quad \text{as }\;k\to\infty.
\end{equation}
Then the assumed estimate \eqref{growthdk} in (ii) leads us to the inequalities
\begin{eqnarray*}
\varphi(x^k+d^k)-\varphi(x^k)-\sigma\langle\nabla\varphi(x^k),d^k\rangle&\le&\langle\nabla\varphi(x^k+d^k),d^k\rangle-\frac{1}{2\kappa}\|d^k\|^2-\sigma\langle\nabla\varphi(x^k),d^k\rangle\\
&\le&\|\nabla\varphi(x^k+d^k)\|\cdot\|d^k\|-\frac{1}{2\kappa}\|d^k\|^2+\sigma\|\nabla\varphi(x^k)\|\cdot\|d^k\|\\
&\le&\ell\|x^{k}+d^k-\bar{x}\|\cdot\|d^k\|-\frac{1}{2\kappa}\|d^k\|^2+\sigma\ell\|x^k-\bar{x}\|\cdot\|d^k\|\\
&\le&\|d^k\|^2 \left( \ell\frac{\|x^{k}+d^k-\bar{x}\|}{\|d^k\|}-\frac{1}{2\kappa}+\sigma\ell\frac{\|x^k-\bar{x}\|}{\|d^k\|}\right)
\end{eqnarray*}
for all large $k\in\N$. Finally, it follows from $\sigma<1/(2\ell\kappa)$, \eqref{superlineark+1}, and \eqref{supdk} that
$$
\varphi(x^k+d^k)-\varphi(x^k)-\sigma\langle\nabla\varphi(x^k),d^k\rangle\le 0\;
\text{ when }\;k\;\mbox{ is sufficiently large}.
$$
This verifies \eqref{backtrackinghold} and thus completes the proof of the lemma.
\end{proof}

\begin{Remark}[\bf on acceptance of unit stepsize] \rm Proposition \ref{acceptancesm} provides a sufficient condition to ensure the {\em asymptotic acceptance} of the unit stepsize. The key assumption here is the semismoothness of $\nabla\varphi$ at the reference point $\ox$, which always includes the {directional differentiability} of $\nabla\varphi$ at $\ox$. Lemma~\ref{Maratos} introduces an alternative condition without the latter property to attain the acceptance of the unit stepsize, which depends on the modulus $\kappa$ in \eqref{growthdk} and the modulus of the Lipschitz continuity of $\nabla\varphi$. The given proof of this result requires the technical condition $\sigma\in(0,1/(2\ell\kappa))$. It is not clear to us whether this condition can be either removed, or replaced by the more simple one $\sigma\in(0,1/2)$. In the case where $\varphi$ is twice differentiable around the critical point $\ox$ and $\nabla^2\varphi$ is continuous at $\ox$, effective sufficient conditions for the acceptance of the unit stepsize can be found in \cite[Section~5.2]{isu}. 
\end{Remark}

Now we are ready to justify the $Q$-superlinear rate of convergence of iterates in Algorithm~\ref{LS} under some additional assumptions and relationships between parameters of the problem and the algorithm.

\begin{Theorem}[\bf superlinear convergence of the coderivative-based damped Newton algorithm in ${\cal C}^{1,1}$ optimization]\label{superlinearLS} In the setting of Theorem~{\rm\ref{globalconver}} ensuring the convergence of $\{x^k\}$ to a {tilt-stable minimizer $\ox$ of $\varphi$} as $k\to\infty$, suppose that $\nabla\varphi$ is locally Lipschitzian around $\bar{x}$ with some constant $\ell>0$ being also semismooth$^*$ at this point. Then the rate of the convergence of $\{x^k\}$ is at least Q-superlinear if either one of the following two conditions is satisfied:

{\bf(i)} $\nabla\varphi$ is directionally differentiable at $\bar{x}$.

{\bf(ii)} $\sigma\in\left(0,1/(2\ell\kappa)\right)$, where $\kappa>0$ is a modulus of tilt stability of $\ox$.\\
Moreover, in both cases {\rm(i)} and {\rm(ii)} the sequence $\{\varphi(x^k)\}$ converges
Q-superlinearly to $\varphi(\bar{x})$, and the sequence $\{\nabla\varphi(x^k)\}$ converges
Q-superlinearly to $0$ as $k\to\infty$.
\end{Theorem}
\begin{proof} Fixing a {tilt-stable} minimizer $\ox$ with modulus $\kappa>0$ from the assertions of Theorem~\ref{globalconver}, we split the proof of this theorem into the three claims.\\[1ex]
{\bf Claim~1:} {\em  {The sequence of directions $\{d^k\}$ satisfies condition \eqref{dksuperlinear}}}. Indeed, by the characterization of tilt-stable minimizers via the combined second-order subdifferential taken from \cite[Theorem~3.5]{MorduNghia} and \cite[Proposition 4.6]{ChieuLee17}, we find a positive number $\delta$ such that the inequality
\begin{equation}\label{uniformPDrate}
\langle z,w\rangle\ge\frac{1}{\kappa}\|w\|^2\quad\text{for all }\;z\in\partial^2\varphi(x)(w),\;x\in\mathbb{B}_\delta(\bar{x}),\;\mbox{ and }\;w\in\R^n
\end{equation}
is satisfied. Employing the subadditivity property of coderivatives taken from \cite[Lemma~5.6]{BorisKhanhPhat} gives us
$$
\partial^2\varphi(x^k)(d^k)\subset\partial^2\varphi(x^k)(x^k+d^k-\bar{x})+\partial^2\varphi(x^k)(-x^k+\bar{x}).
$$
Since $-\nabla\varphi(x^k)\in\partial^2\varphi(x^k)(d^k)$, for all $k\in\N$ there exists $v^k\in\partial^2\varphi(x^k)(-x^k+\bar{x})$ such that
$$
-\nabla\varphi(x^k)-v^k\in\partial^2\varphi(x^k)(x^k+d^k-\bar{x}).
$$
Using further \eqref{uniformPDrate} and the Cauchy-Schwarz inequality, we get
\begin{equation}\label{semi}
\|x^k+d^k-\bar{x}\|\le\kappa\|\nabla\varphi(x^k)+v^k\|\quad\text{for sufficiently large }\;k\in\N.
\end{equation}
The semismoothness$^*$ of $\nabla\varphi$ at $\bar{x}$ together with $\nabla\varphi(\bar{x})=0$ implies by \cite[Lemma~5.5]{BorisKhanhPhat} that
\begin{equation}\label{semi2}
\|\nabla\varphi(x^k)+v^k\|=\|\nabla\varphi(x^k)-\nabla\varphi(\bar{x})+v^k\|=o(\|x^k-\bar{x}\|).
\end{equation}
Then it follows from \eqref{semi} and \eqref{semi2} that $\|x^k+d^k-\bar{x}\|=o(\|x^k-\bar{x}\|)$, which justifies the claim.\\[1ex]
{\bf Claim~2:} {\em We have $\tau_k=1$ for all $k\in\N$ sufficiently large provided that either condition {\rm(i)}, or condition {\rm(ii)} of this theorem is satisfied}. To verify the claim, let us show that \eqref{backtrackinghold} holds for large $k$ under the imposed assumptions. Suppose first that (i) is satisfied. {The directional differentiability and semismoothness$^*$ of $\nabla \varphi$ at $\bar{x}$ ensure by \cite[Corollary~3.8]{Helmut} that $\nabla\varphi$ is semismooth at
$\bar{x}$}. Due to the inclusion $-\nabla\varphi(x^k)\in\partial^2\varphi(x^k)(d^k)$ and estimate \eqref{uniformPDrate}, we have that $\langle\nabla\varphi(x^k),d^k\rangle\le-\kappa^{-1}\|d^k\|^2$ if $k$ is large enough. Then the fulfillment of \eqref{backtrackinghold} in case (i) follows directly from  {Proposition~\ref{acceptancesm}}. In case (ii), we know from Claim~1 that $\{d^k\}$ converges to $0$ and that $x^k+d^k\to\bar{x}$ as $k \to\infty$. Employing the uniform second-order growth condition for tilt-stable minimizers from \cite[Theorem~3.2]{MorduNghia} gives us a neighborhood $U$ of $\bar{x}$ such that
\begin{equation*}
\varphi(x)\ge\varphi(u)+\langle\nabla\varphi(u),x-u\rangle+\frac{1}{2\kappa}\|x-u\|^2\quad\text{for all }\;x,u\in U,
\end{equation*}
and thus verifies \eqref{growthdk}. Using Lemma~\ref{Maratos} brings us to \eqref{backtrackinghold}, which verifies the claim.\\[1ex]
{\bf Claim~3:} {\em The conclusions on the Q-superlinear convergence in the theorem hold in both cases ${\rm(i)}$ and ${\rm(ii)}$}. We see from Claim~2 that $\tau_k=1$ for all $k$ sufficiently large, and thus Algorithm~\ref{LS} eventually becomes the generalized pure Newton algorithm from \cite[Algorithm~5.3]{BorisKhanhPhat}. Hence the claimed $Q$-superlinear convergence results follow from \cite[Theorems~5.7 and 5.12]{BorisKhanhPhat}.\end{proof}\vspace*{-0.15in}

 \begin{Remark}[\bf comparing Algorithm~\ref{LS} with other Newton-type methods]\label{compareremark1} \rm There exist several generalized Newton-type methods providing superlinearly convergent iterates under appropriate assumptions; see \cite{JPang,Solo14,Klatte} and the bibliographies therein. We briefly compare Algorithm~\ref{LS} with the two most popular globalized Newtonian methods. The first one is known as the {\em semismooth Newton method} initiated independently by Kummer \cite{ku} and by Qi and Sun \cite{qs}. The second method was introduced by Pang \cite{Pang90} under the name of the {\em B-differential Newton method}. Both methods address solving Lipschitzian equations $f(x)=0$, which reduce in the setting of \eqref{opproblem} to finding solutions of the gradient equation
\begin{equation}\label{equationNM}
\nabla\varphi(x)=0,\quad x\in\R^n.
\end{equation}
 Regarding the aforementioned methods to solve the gradient equation \eqref{equationNM}, observe the following:

{\bf(i)}   The semismooth Newton method and its globalizations for solving \eqref{equationNM} are based on Clarke's generalized Jacobian $\partial_C \nabla\ph$ of $\nabla\ph$ (see below) to find Newton directions $d^k$ as solutions to the system of linear equations 
\begin{equation}\label{semi-dir}
-\nabla\varphi(x^k)=A^kd, 
\end{equation}
where $A^k $ is an element of $\partial_C \nabla\ph(x^k)$ as the convex hull of the \textit{Bouligand's Jacobian}
\begin{equation}\label{lim-hes}
\partial_B \nabla \ph(x):=\Big\{\lim_{m\to\infty}\nabla^2\ph(u_m)\;\Big|\;u_m\to
x,\;u_m\in Q_\ph\Big\},\quad x\in\R^n, 
\end{equation} 
with $Q_\ph$ standing for the set on which $\ph$ is twice differentiable. 
A strong feature of \eqref{semi-dir} is the {\em linearity} of equations therein, although the solvability of these equations requires the {\em nonsingularity} of all the matrices $A^k$.  {Moreover, computation cost to solve the system of linear equations \eqref{semi-dir} is known to be expensive.}  Although system \eqref{damped-dir} for finding directions in Algorithm~\ref{LS} is not linear, we get from it a smaller set of algorithm directions in comparison with \eqref{semi-dir}.  {The detailed numerical experiment to compare our approach and semismooth Newton method in a specific $\mathcal{C}^{1,1}$ optimization problem can be found in Section \ref{Lassosec}}. Another advantage of Algorithm~\ref{LS} over the semismooth Newton method is well-developed second-order subdifferential calculus that is not available for \eqref{lim-hes} and its convexification. 
Theoretical comparisons of local convergence between \textit{coderivative-based Newtonian methods} and semismooth Newton methods can be found in \cite{BorisKhanhPhat,BorisEbrahim}, where the regularity assumptions and the nonsingularity of all matrices in the generalized Jacobian have been discussed in detail. Regarding the global convergence   of our approach and the semismooth Newton method,  observe from \cite[Theorem 13.52]{Rockafellar98} that 
\begin{equation}
    \co \partial^2 \varphi(x) (w) = \co \{Aw|\; A \in \partial_B\nabla \varphi (x) \}, \quad x \in \R^n, w\in\R^n,
\end{equation}
which tells us that the positive-definiteness of $\partial^2\varphi(x)$ is equivalent to the positive-definiteness of $\partial_C \nabla \varphi(x)$ and $\partial_B \nabla \varphi(x)$. This positive-definiteness is required for both global convergence of  Algorithm \ref{LS} and globalization of semismooth Newton method. It happens because we not only need the existence of generalized Newton directions in inclusion  \eqref{damped-dir} and in the system of linear equations \eqref{semi-dir}, but the descent property of these directions is also needed to achieve the desired global convergence.
\color{black}

{\bf(ii)}  {The B-differential Newton method for solving equation \eqref{equationNM} developed in \cite{Pang90} and \cite{Qi93} is based on Robinson's B-derivative, which reduces for Lipschitzian mappings to the classical directional derivative. For this method, we need  to solve the subproblem given below to find Newton directions $d^k$ as solutions to
\begin{equation}\label{B-diff}
-f(x^k) =  f^\prime (x^k,d^k), \quad \text{where }\; f:=\nabla \varphi,
\end{equation}
with requiring the directional differentiability of $\nabla\varphi$. As shown in \cite{Pang90}, the main assumptions to guarantee the solvability of \eqref{B-diff} are the strict Fr\'echet differentiability of $f$ and the nonsingularity of Jacobian $\nabla f$ at the solution point, which are rather restrictive requirements in comparison with our assumptions.}
\end{Remark}
\vspace*{-0.2in}

\section{Coderivative-Based Regularized Newton Method}\label{sec:modifiedC11}\vspace*{-0.05in}

\color{black} Observe that the positive-definiteness of the generalized Hessian $\partial^2\varphi(x)$ in Algorithm~\ref{LS} cannot be replaced by the less demanding positive-semidefiniteness of $\partial^2\ph(x)$ to ensure the existence of descent Newton direction in Algorithm~\ref{LS} as in Proposition~\ref{descent}. Indeed, consider the simplest linear function $\varphi(x):=x$ on $\R$. Then we obviously have that $\varphi''(x)\ge 0$ for all $x\in\R$, while there are no Newton directions $d\in\R$  {such that the backtracking line search condition \eqref{Armijo} holds.} This means that Algorithm~\ref{LS} cannot be even constructed without the positive-definiteness of $\partial^2\varphi(x)$. Here we propose the following globally convergent coderivative-based {\em   {generalized regularized Newton algorithm}} to solve problems of ${\cal C}^{1,1}$ optimization that is {\em well-posed} and exhibits the {\em convergence} of its subsequences to {\em stationary points} of $\ph$ under merely the {\em positive-semidefiniteness} of the generalized Hessian. Linear and superlinear {\em convergence rates} are achieved under some additional assumptions.

\begin{algorithm}[H]
\caption{{Coderivative-based regularized Newton algorithm for $\mathcal{C}^{1,1}$ functions}}\label{LS2}
\begin{algorithmic}[1]
\Require {$x^0\in\R^n$, $c>0$, $\sigma\in\left(0,\frac{1}{2}\right)$, $\beta\in\left(0,1\right)$}
\For{\texttt{$k=0,1,\ldots$}}
\State\text{If $\nabla\varphi(x^k)=0$, stop; otherwise let $\mu_k:= c\|\nabla\varphi(x^k)\|$ and go to next step}
\State \label{alC112} \text{Choose $d^k\in\R^n$ such that
$-\nabla\varphi(x^k)\in\partial^2\varphi(x^k)(d^k) + \mu_k d^k$}
\State Set $\tau_k = 1$
\While{$\varphi(x^k+\tau_kd^k)>\varphi(x^k)+\sigma\tau_k\langle\nabla\varphi(x^k),d^k\rangle$}
\State \text{set $\tau_k:= \beta \tau_k$}
\EndWhile\label{euclidendwhile2}
\State \text{Set $x^{k+1}:=x^k+\tau_k d^k$}
\EndFor
\end{algorithmic}
\end{algorithm}

Our first major result in this section establishes the well-posedness and global convergence of iterates generated by Algorithm~\ref{LS2} to stationary points of $\ph$ under only the positive-semidefiniteness assumption on the generalized Hessian $\partial^2\varphi(x)$.

\begin{Theorem}[\bf {well-posedness and convergence of the
coderivative-based regularized Newton algorithm}]
\label{limitingpoint}  {For a function $\varphi\colon\R^n\to\R$ of class
$\mathcal{C}^{1,1}$ on $\R^n$, the following assertions hold: 

{\bf(i)} Let $x \in \R^n$ be such that $\nabla\varphi(x)\ne 0$
and $\partial^2\varphi(x)$ is
positive-semidefinite, i.e., 
\begin{equation}\label{pos-def2}
\langle z,w\rangle\ge 0\;\mbox{ for all
}\;z\in\partial^2\varphi(x)(w)\;\mbox{ and }\;w\in\R^n.
\end{equation} 
Then for any $\epsilon>0$, there exists a nonzero direction
$d\in\R^n$ with 
\begin{equation}\label{descentforvaphi2}
-\nabla\varphi(x)\in\partial^2\varphi(x)(d) +\epsilon d.
\end{equation} 
Moreover, every such direction satisfies the
inequality $\langle \nabla\varphi(x),d\rangle<0$. Consequently, for
each $\sigma\in(0,1)$ and $d\in\R^n$  {satisfying
\eqref{descentforvaphi2}} we have $\delta>0$ such that
\begin{equation}\label{Armijo2} \varphi(x+\tau
d)\le\varphi(x)+\sigma\tau\langle\nabla\varphi(x),d\rangle\;\mbox{
whenever }\;\tau\in(0,\delta). \end{equation} 

{\bf(ii)} Picking any starting point $x^0\in\R^n$ such that $\partial^2\varphi(x)$ is
positive-semidefinite for all $x\in \O$ from \eqref{level}, we
have that Algorithm~{\rm\ref{LS2}} either stops after finitely many
iterations, or produces a sequence of iterates
$\{x^k\}\subset\Omega$ such that the sequence of values
$\{\varphi(x^k)\}$ is monotonically decreasing. Moreover, all the
accumulation points of $\{x^k\}$ satisfy the stationarity condition.}
\end{Theorem} 
\begin{proof} To justify (i), fix $x$
satisfying the assumptions therein and consider the function
\begin{equation*}
\varphi_\epsilon(\cdot)=\varphi(\cdot)+\frac{\epsilon}{2}\|\cdot
\|^2\;\mbox{ for  any }\;\ve>0 \;\mbox{ on }\;\R^n. \end{equation*}
It follows from the second-order subdifferential sum rule in
\cite[Proposition~1.121]{Mordukhovich06} that
\begin{equation}\label{sumrulecod}
\partial^2\varphi_\epsilon(x)(w)=\partial^2\varphi(x)(w)+\epsilon
w\quad\text{for all }\;w\in \R^n. \end{equation} Thus $z-\epsilon w
\in \partial^2\varphi(x)(w)$ whenever
$z\in\partial^2\varphi_\epsilon(x)(w)$. Due to the
positive-semidefiniteness of $\partial^2\varphi(x)(w)$, we get
$\langle z, w\rangle\ge\epsilon\|w\|^2$, which implies that
$\partial^2\varphi_\epsilon(x)$ is positive-definite. It follows
from \cite[Theorem~5.16]{Mor18} that $\nabla\varphi_\epsilon$ is
locally strongly maximally monotone around
$(x,\nabla\varphi_\epsilon(x))$. Hence the gradient mapping $\nabla
\varphi_\epsilon$ is strongly metrically regular around this point
due to \cite[Theorem~5.13]{Mor18} telling us that the inverse
mapping $\nabla\varphi_\epsilon^{-1}:\R^n \rightrightarrows \R^n$
admits a single-valued localization $\vartheta: V\to U$ around
$(\nabla\varphi_\epsilon(x),x)$, which is locally Lipschitzian
around the point $\nabla\varphi_\epsilon(x)$. Combining this with
the scalarization formula \eqref{scal} yields the representations
\begin{equation}\label{scalarcode} D^*\nabla
\varphi_\epsilon^{-1}\big(\nabla\varphi_\epsilon(x),x\big)\big(\nabla
\varphi(x)\big)=
D^*\vartheta\big(\nabla\varphi_\epsilon(x)\big)\big(\nabla
\varphi(x)\big)= \partial\big\langle
\nabla\varphi(x),\vartheta\big\rangle\big(\nabla
\varphi_\epsilon(x)\big). \end{equation} Since $\vartheta$ is
locally Lipschitzian, we deduce from
\cite[Theorem~1.22]{Mordukhovich06} that
$\partial\langle\nabla\varphi_\epsilon(x),\vt\ra(\nabla\ph_\ve(x))\ne\emp$.
Thus  {it follows from \eqref{scalarcode} that} \begin{equation}\label{inv}
D^*\nabla\varphi_\epsilon^{-1}\big(\nabla\varphi_\epsilon(x),x\big)\big(\nabla
\varphi(x)\big)\ne\emp. \end{equation} Picking any $-d\in
D^*\nabla\varphi_\epsilon^{-1}(\nabla\varphi_\epsilon(x),x)(\nabla\varphi(x))$
and easily representing the coderivative of the inverse mapping
$\nabla\ph_\ve^{-1}$ via that of $\nabla\ph_\ve$, we deduce from
\eqref{inv} the inclusion \begin{equation*} -\nabla\varphi(x)\in
D^*\nabla\varphi_\epsilon (x)(d)=\partial^2\varphi_\epsilon(x)(d).
\end{equation*} Due to \eqref{sumrulecod}, it follows from the above
that $-\nabla\varphi(x)\in\partial^2\varphi(x)(d)+\epsilon d$.

To verify (i), it remains to show that $d\ne 0$. Supposing the contrary and using the local Lipschitz continuity of $\nabla\varphi$, we obtain from \cite[Theorem~1.44]{Mordukhovich06} that
$$
-\nabla\varphi(x)\in\partial^2\varphi(x)(0)=\big(D^*\nabla\varphi\big)(x)(0)=\{0\},
$$
which contradicts the imposed assumption $\nabla\varphi(x)\ne 0$. The
positive-definiteness of $\partial^2\varphi_\epsilon(x)$ yields $\langle\nabla\varphi(x),d\rangle<0$ that ensures in turn the {fulfillment of \eqref{Armijo2}} due to \cite[Lemmas~2.18 and 2.19]{Solo14}.\vspace*{0.03in}

Next we proceed with the proof of (ii). It follows from (i) while arguing by induction that Algorithm~\ref{LS2} either stops after finitely many iterations, or produces a sequence of iterates $\{x^k\}\subset\Omega$ such that $\varphi(x^{k+1})<\varphi(x^k)$ for all $k\in\N$. Let us first show that the sequence $\{d^k\}$ is bounded. Indeed, it follows from the construction that
\begin{equation}\label{inclusiond}
-\nabla\varphi(x^k)-\mu_k d^k\in\partial^2\varphi(x^k)(d^k)\quad\text{for all }\;k\in\N,
\end{equation}
and thus we get that $\langle-\nabla\varphi(x^k)-\mu_k d^k, d^k\rangle \ge 0$, i.e.,
\begin{equation}\label{PD}
\langle\nabla\varphi(x^k),-d^k\rangle\ge\mu_k\|d^k\|^2,\quad k\in\N.
\end{equation}
Employing the Cauchy-Schwarz inequality and replacing $\mu_k$ with $c\|\nabla\varphi(x^k)\|$ lead us to
$$
c\|\nabla \varphi(x^k)\|\cdot\|d^k\|=\mu_k\|d^k\|\le\|\nabla \varphi(x^k)\|
$$
and readily implies that $\|d^k\|\le 1/c$ for all $k$.

Fix now  {an accumulation point} $\ox$ of the sequence of iterates $\{x^k\}$ and find a subsequence $\{x^{k_j}\}$ of $\{x^k\}$ such that $x^{k_j}\to \bar{x}$ as $j\to\infty$. Since the sequence $\{\varphi(x^k)\}$ is nonincreasing and $\varphi(\bar{x})$ is  {an accumulation point} of $\{\varphi(x^k)\}$, this sequence converges to $\varphi(\bar{x})$ as $k\to\infty$. Moreover, we have
$$
\varphi(x^{k+1})-\varphi(x^k)\le\sigma\tau_k\langle\nabla\varphi(x^k),d^k\rangle<0\quad \text{for all }\;k\in\N,
$$
which yields the equality
\begin{equation}\label{gradientdirection}
\lim_{k\to\infty}\tau_k\langle \nabla\varphi(x^k),d^k\rangle=0.
\end{equation}
Using the boundedness of $\{d^k\}$, we find a subsequence $\{d^{k_j}\}$ converging to some $\bar{d}\in\R^n$. Let us verify that
\begin{equation}\label{gradientdirection2}
\left\langle \nabla \varphi(\bar{x}),\bar{d}\right\rangle = 0.
\end{equation}
Indeed, if $\displaystyle\limsup_{j \to \infty}\tau_{k_j}>0$, then \eqref{gradientdirection2} follows immediately from \eqref{gradientdirection}. Otherwise, we have $\displaystyle \lim_{j \to \infty}\tau_{k_j}=0$, and the exit condition of the backtracking line search in Step~5 of Algorithm~\ref{LS2} brings us to
\begin{equation}\label{exitcond2}
\varphi\big(x^{k_j}+\tau_{k_j}^\prime d^{k_j}\big)>\varphi\big(x^{k_j}\big)+\sigma\tau_{k_j}^\prime\langle\nabla \varphi(x^{k_j}),d^{k_j}\rangle
\end{equation}
for all $j\in\N$, where $\tau_{k_j}^\prime :=\tau_{k_j}/\beta$. Dividing now both sides of \eqref{exitcond2} by $\tau_{k_j}^\prime$ and letting $j\to\infty$ implies that
$$
\left\langle\nabla\varphi(\bar{x}),\bar{d}\right\rangle =\lim_{j\to\infty}\frac{\varphi\big(x^{k_j}+\tau_{k_j}^\prime d^{k_j}\big)-\varphi\left(x^{k_j}\right)}{\tau_{k_j}^\prime}\ge\sigma\left\langle \nabla\varphi(\bar{x}),\bar{d}\right\rangle.
$$
This tells us that $\left\langle\nabla\varphi(\bar{x}),\bar{d}\right\rangle\ge 0$ by $\sigma<1$. Letting $j\to\infty$ in $\left\langle \nabla\varphi(x^{k_j}),d^{k_j}\right \rangle\le 0$, we get $\left\langle\nabla\varphi(\bar{x}),\bar{d}\right\rangle\le 0$ and arrive at \eqref{gradientdirection2}. Combining \eqref{PD} and \eqref{gradientdirection2} verifies that $\mu_{k_j}\|d^{k_j}\|^2 \to 0$ as $j\to\infty$. By the definition of $\{\mu_k \}$ and the convergence $x^{k_j}\to\ox$, we have $\mu_{k_j}= c\|\nabla\varphi(x^{k_j})\|\to c\|\nabla\varphi(\ox)\|$ as $j\to\infty$, which ensures that
\begin{equation}\label{limestimate}
c\left\|\nabla \varphi(\bar{x})\right\|\cdot\left\|\bar{d}\right\|^2=\lim_{j\to\infty}\mu_{k_j}\left\|d^{k_j}\right\|^2 =0.
\end{equation}
Since $\ph$ is of class ${\cal C}^{1,1}$ around $\bar{x}$, it follows from  \cite[Theorem~1.44]{Mordukhovich06} and \eqref{inclusiond} that there exists $\ell>0$ such that
$\|\nabla\varphi(x^{k_j}) + \mu_{k_j} d^{k_j}\|\le\ell \|d^{k_j}\|$
for all $j$ sufficiently large. This yields
$$
\left\| \nabla \varphi(x^{k_j})\right\|^2 + 2\mu_{k_j}\left\langle\nabla \varphi(x^{k_j}), d^{k_j}\right\rangle+\mu_{k_j}^2\left\|d^{k_j}\right\|^2\le\ell^2\left\|d^{k_j}\right\|^2
$$
for such $j$. Letting $j\to\infty$ in the above inequality, we arrive at $\left\|\nabla \varphi(\bar{x})\right\|^2\le\ell^2\left\|\bar{d}\right\|^2$ due to \eqref{gradientdirection2} and the second equality in \eqref{limestimate}. Using the obtained estimate together with the first part of  \eqref{limestimate} gives us $\nabla\varphi(\bar{x})=0$ and thus completes the proof of the theorem.
\end{proof}\vspace*{-0.05in}

\begin{Remark}[\bf on proof of Theorem~\ref{limitingpoint}(ii)]\rm  {Following the suggestion of the referee, we provide an alternative proof of the assertions in (ii) of Theorem~\ref{limitingpoint} by assuming the contrary and using the {\em gradient related property} taken from \cite{Bert}. Indeed, suppose that $\nabla\varphi(\ox)\ne 0$ and then show that $\{d^k\}$ is gradient related to 
$\{x^k\}$, i.e., for any subsequence $\{x^{k_j}\}$ converging to $\ox$, the corresponding subsequence $\{d^{k}_{j}\}$ is bounded and satisfies
\begin{equation}\label{gradientrelated2}
\limsup_{j\to\infty}\langle\nabla\varphi(x^{k_j}),d^{k_j}\rangle<0. 
\end{equation}
Indeed, the boundedness of $\{d^k\}$ is verified by the same proof as in Theorem~\ref{limitingpoint}(ii). 
Due to \eqref{PD}, we have 
\begin{equation}\label{gradrelatLM}
\langle\nabla\varphi(x^{k_j}),d^{k_j}\rangle\le-\mu_{k_j}\|d^{k_j}\|^2=-c\|\nabla\varphi(x^{k_j})\|.\|d^{k_j}\|^2\;\text{for all }\;j\in \N. 
\end{equation}
Since $\ph$ is of class ${\cal C}^{1,1}$ around $\bar{x}$, it follows from  \cite[Theorem~1.44]{Mordukhovich06} and \eqref{inclusiond} that there exists $\ell>0$ such that
$\|\nabla\varphi(x^{k_j})+\mu_{k_j}d^{k_j}\|\le\ell\|d^{k_j}\|$
for all $j$ sufficiently large. By the definition of $\{\mu_k \}$ and the convergence $x^{k_j}\to\ox$, we have that $\mu_{k_j}=c\|\nabla\varphi(x^{k_j})\|\to c\|\nabla\varphi(\ox)\|$ as $j\to\infty$, which ensures that $\mu_{k_j}$ is bounded from above by some $m>0$, and thus arrive at the estimate
\begin{equation}\label{dknablavaphi}
\|\nabla\varphi(x^{k_j})\|\le\|\nabla\varphi(x^{k_j})+\mu_{k_j}d^{k_j}\|+\mu_{k_j}\|d^{k_j}\|\le(\ell+m)\|d^{k_j}\|\;\text{ for all }\;j\in \N. 
\end{equation} 
Combining \eqref{gradrelatLM} and \eqref{dknablavaphi} tells us that
$$
\langle\nabla\varphi(x^{k_j}),d^{k_j}\rangle\le-c(\ell+m)^{-2}\|\nabla\varphi(x^{k_j})\|^3\;\text{ for all }\;j\in\N.
$$
By taking the upper limit in both sides above and using $\nabla\varphi(\ox)\ne 0$ justifies \eqref{gradientrelated2}.  Therefore, it follows from \cite[Proposition~1.2.1]{Bert} that $\nabla \varphi(\ox)=0$, a contradiction that verifies {\bf (ii)}.}
\end{Remark}

The next theorem establishes the linear and superlinear {\em
convergence rates} of iterates in Algorithm~\ref{LS2} to {\em
tilt-stable minimizers} under the {\em metric regularity assumption}
on $\nabla\ph$ at the solution point $\ox$. Note that the latter
property is constructively characterized by \eqref{Morcri} and
\eqref{scal1} as \begin{equation*}
\big\{u\in\R^n\;\big|\;0\in\partial\big\la
u,\nabla\ph\big\ra(\ox)\big\}=\{0\}. 
\end{equation*}

\begin{Theorem}[\bf {linear and superlinear global convergence of coderivative-based regularized Newton algorithm}]\label{linearcon2} In the setting of Theorem~{\rm\ref{limitingpoint}}, let $\bar{x}$ be  {an accumulation point} of $\{x^k\}$ such that $\nabla\varphi$ is metrically regular around this point. Then $\ox$ is a tilt-stable local minimizer of $\varphi$ and Algorithm~{\rm\ref{LS2}} converges to $\bar{x}$ with the convergence rates as follows:

{\bf(i)} The sequence of values $\{\varphi(x^k)\}$ converges to $\ph(\ox)$ at least Q-linearly.

{\bf(ii)} The sequences $\{x^k\}$ and $\{\nabla\varphi(x^k)\}$ converge at least R-linearly to $\ox$ and $0$, respectively.

{\bf(iii)} The convergence rates of $\{x^k\}$, $\{\varphi(x^k)\}$, and $\{\nabla\varphi(x^k)\}$ are at least Q-superlinear if
$\nabla\varphi$ is semismooth$^*$ at $\bar{x}$ and either one of the following two conditions holds:

{\bf(a)} $\nabla\varphi$ is directionally differentiable at $\bar{x}$.

{\bf(b)} $\sigma\in\left(0,1/(2\ell\kappa)\right)$, where $\kappa>0$ and $\ell>0$ are moduli of metric regularity and Lipschitz
continuity of $\nabla\varphi$ around $\ox$, respectively.
\end{Theorem}
\begin{proof} We split the proof into the seven major claims of their own interest.\\[1ex]
{\bf Claim~1:} {\em  $\bar{x}$  is a tilt-stable local minimizer of $\varphi$.} Due to Theorem \ref{limitingpoint}, $\bar{x}$ is a stationary point of $\varphi$ and $\bar{x}\in\Omega$, which implies that $\partial^2\varphi(\bar{x})$ is positive-semidefinite. This property and the imposed metric regularity of $\nabla\varphi$ around $\bar{x}$ with modulus $\kappa$ allow us to conclude by using \cite[Theorem~4.13]{dmn} that $\bar{x}$ is a tilt-stable local minimizer of $\varphi$ with the same modulus $\kappa$. \\[1ex]
{\bf Claim~2:} {\em  For any subsequence $\{x^{k_j}\}$ of $\{x^k\}$ with $x^{k_j}\to \bar{x}$ as $j\to\infty$, the corresponding sequence $\{\tau_{k_j}\}$ in Algorithm~{\rm\ref{LS2}} is bounded from below by a positive number $\gamma$, and we have}
\begin{equation}\label{xk+1(2)}
\varphi(x^{k_j})-\varphi(x^{k_j+1})\ge\frac{\sigma\gamma}{\kappa} \|d^{k_j}\|^2\quad \text{for all large}\;j\in \N.
\end{equation}
Indeed, supposing on the contrary that $\{\tau_{k_j}\}$ is not bounded from below by a positive number and combining this with $\tau_k\ge 0$ give us a subsequence of $\{\tau_{k_j}\}$ that converges to $0$. Let $\tau_{k_j}\to 0$ as $j\to\infty$ without loss of generality. Then using the characterization of tilt-stable minimizers via the combined second-order subdifferential taken from \cite[Theorem~3.5]{MorduNghia} and \cite[Proposition~4.6]{ChieuLee17}, we find $\delta>0$ with
\begin{equation}\label{uniformPSD}
\langle z,w\rangle\ge\frac{1}{\kappa}\|w\|^2\quad\text{for all }\;z\in\partial^2\varphi(x)(w),\;x\in\mathbb{B}_\delta(\bar{x}),\;\mbox{ and }\;w\in\R^n.
\end{equation}
Since $-\nabla \varphi(x^{k_j})-\mu_{k_j}d^{k_j}\in\partial^2\varphi(x^{k_j})(d^{k_j})$ for all $j\in\N$, it follows from \eqref{uniformPSD} that
\begin{equation}\label{dk2}
\langle-\nabla \varphi(x^{k_j}), d^{k_j}\rangle\ge\left(\mu_{k_j}+\frac{1}{\kappa}\right) \|d^{k_j}\|^2 \ge\frac{1}{\kappa}\|d^{k_j}\|^2 \quad \text{for all $j$ sufficiently large}.
\end{equation}
Then Claim~1 of Theorem~\ref{limitingpoint} tells us that the sequence $\{d^k\}$ is bounded.
Hence $x^{k_j}+\tau_{k_j}d^{k_j}\to\bar{x}$ as $j\to\infty$ and $
x^{k_j}+\tau_{k_j}d^{k_j}\in\text{\rm int}\,\mathbb{B}_\delta(\bar{x})$ whenever $j$ is sufficiently large. Applying  Lemma~\ref{estimate1} from the Appendix, we get 
$$
\beta^{-1}\tau_{k_j} > \frac{2(\sigma -1)\langle \nabla \varphi(x^{k_j}),d^{k_j}\rangle}{\ell\|d^{k_j}\|^2} \geq \frac{2(1-\sigma)}{\kappa\ell},
$$
where the second inequality follows from \eqref{dk2}.  
\color{black}
Letting $j\to\infty$ gives us $\sigma\ge 1$, a contradiction due to $\sigma<1$. This verifies the existence of $\gamma>0$ such that $\tau_{k_j}\ge\gamma$ for all $j\in\N$. Using estimate \eqref{dk2}, we find $j_0\in\N$ with
\begin{equation}\label{ineq3}
\varphi(x^{k_j})-\varphi(x^{k_j+1})\ge\sigma\tau_{k_j}\langle-\nabla\varphi(x^{k_j}),d^{k_j}\rangle\ge\frac{\sigma\gamma}{\kappa}\|d^{k_j}\|^2 \quad \text{for all }\;j\ge j_0,
\end{equation}
which therefore justifies Claim~2.\\[1ex]
{\bf Claim~3:} {\em   {The iterative sequence $\{x^k\}$ converges to $\ox$}.} To verify this, we are based on Ostrowski's condition from \cite[Proposition~8.3.10]{JPang}. Let us first check that there is no other  {accumulation point} of $\{x^k\}$ in $\mathbb{B}_\delta(\bar{x})$. On the contrary,
suppose that there exists $\Tilde x\in\mathbb{B}_\delta(\bar{x})$ such that $\Tilde x\ne\bar{x}$ and $\Tilde x$ is  {an accumulation point} of $\{x^k\}$. It follows from Theorem~\ref{limitingpoint} that $\Tilde x$ is a stationary point of $\varphi$, which contradicts the strong convexity of $\varphi$ on $\mathbb{B}_\delta(\bar{x})$. Supposing next that $\{x^{k_j}\}$ is an arbitrary subsequence of $\{x^k\}$ with $x^{k_j}\to\bar{x}$ as $j\to\infty$, we check that
\begin{equation}\label{Ostrowski2}
\lim_{j\to\infty}\|x^{k_j+1}-x^{k_j}\|=0.
\end{equation}
Indeed, find by Claim~2 such $\gamma>0$ that \eqref{xk+1(2)} holds,  which implies that
$$
\|x^{k_j+1}- x^{k_j}\|^2= \tau_{k_j}^2 \|d^{k_j}\|^2 \le\|d^{k_j}\|^2 \leq \frac{\kappa}{\sigma\gamma}\left(\varphi(x^{k_j})-\varphi(x^{k_j+1}) \right) \to 0
$$
as $j\to\infty$ and thus verifies \eqref{Ostrowski2}. Employing now
\cite[Proposition~8.3.10]{JPang} ensures the convergence of $\{x^k\}$ to $\bar{x}$ as $k\to\infty$ and therefore completes the proof of this claim.\\[1ex]
{\bf Claim~4:} {\em The convergence rate of $\{\varphi(x^k)\}$ is at least Q-linear, while the convergence rates of $\{x^k\}$ and $\{\|\nabla\varphi(x^k)\| \}$ are at least R-linear.} Indeed, the strong convexity of $\varphi$ on $\mathbb{B}_\delta(\bar{x})$ implies that
\begin{equation}\label{second-growth2}
\varphi(x)\ge\varphi(u)+\langle\nabla\varphi(u),x-u\rangle+\frac{1}{2\kappa}\|x-u\|^2\;\mbox{ and }\;\langle\nabla\varphi(x)-\nabla \varphi(u),x-u\rangle\ge\frac{1}{\kappa}\|x-u\|^2
\end{equation}
for all $x,u\in\mathbb{B}_\delta(\bar{x})$. By the convergence $x^k\to\bar{x}$ we have that $x^k\in U$ for all $k$ sufficiently large, which we assumed from now on. Substituting $x:=x^k$ and $u:=\bar{x}$ into \eqref{second-growth2} and then using the Cauchy-Schwarz inequality together with $\nabla\varphi(\bar{x})=0$ yield the estimates
\begin{equation}\label{second-growth12}
\varphi(x^k)\ge\varphi(\bar{x})+\frac{1}{2\kappa}\|x^k-\bar{x}\|^2\;\mbox{ and}
\end{equation}
\begin{equation}\label{strongineq2}
\|\nabla\varphi(x^k)\|\ge\frac{1}{\kappa}\|x^k-\bar{x}\|.
\end{equation}
The local Lipschitz continuity of $\nabla\varphi$ around $\bar{x}$ and the result of \cite[Lemma~A.11]{Solo14} ensure the existence of a positive number $\ell$ such that
\begin{equation}\label{Lipschitzineq2}
\varphi(x^k)-\varphi(\bar{x})=|\varphi(x^k)-\varphi(\bar{x})-\langle\nabla\varphi(\bar{x}),x^k-\bar{x}\rangle|\le\frac{\ell}{2}\|x^k-\bar{x}\|^2.
\end{equation}
Moreover, since $-\nabla\varphi(x^{k}) -\mu_k d^k\in\partial^2\varphi(x^{k})(d^{k})$, by using \cite[Theorem~1.44]{Mordukhovich06} we have
\begin{equation}\label{gradientell}
\|\nabla \varphi(x^k) + \mu_k d^k\|\le\ell \|d^k\|.
\end{equation}
It follows from the convergence $x^k \to \bar{x}$ and $\nabla\varphi(\bar{x})=0$ that $\mu_k = c\|\nabla\varphi(x^k)\|\to 0$ as $k\to\infty$, which implies that $\mu_k \le \ell$. Combining the latter with \eqref{gradientell} gives us the estimates
\begin{equation}\label{Lip}
\|\nabla \varphi(x^k)\| \le\|\nabla \varphi(x^k) + \mu_k d^k\| + \mu_k \|d^k\|\le 2\ell \|d^k\|.
\end{equation}
By Claim~2 we have that $\{\tau_k\}$ is bounded from below by some constant $\gamma>0$ and that
$$
\varphi(x^{k})-\varphi(x^{k+1})\ge\frac{\sigma\gamma}{\kappa}\|d^{k}\|^2,
$$
which together with \eqref{Lip} yields the inequality
\begin{equation}\label{ineq2+}
\varphi(x^{k})-\varphi(x^{k+1})\ge\frac{\sigma\gamma}{4\kappa\ell^2}\|\nabla\varphi(x^k) \|^2.
\end{equation}
Combining finally \eqref{strongineq2}, \eqref{Lipschitzineq2}, and \eqref{ineq2+} and then applying Lemma~\ref{QRlinear} with the sequences $\alpha_k:=\varphi(x^k)-\varphi(\bar{x})$, $\beta_k:= \|\nabla\varphi(x^k)\|$, $\gamma_k:= \|x^k-\bar{x}\|$ and positive numbers $c_1:= (\sigma \gamma)/(4\kappa\ell^2)$, $c_2:= 1/\kappa$, and $c_3:=\ell/2$, we complete the verification of all the conclusions of this claim.
\\[1ex]
{\bf Claim~5:} {\em  {$\|x^k+d^k-\ox\|=o(\|x^k-\ox\|)$} provided that $\nabla\varphi$ is semismooth$^*$ at $\bar{x}$.} Indeed, the subadditivity property of coderivatives taken from \cite[Lemma~5.6]{BorisKhanhPhat} tells us that
$$
\partial^2\varphi(x^k)(d^k)\subset\partial^2\varphi(x^k)(x^k+d^k-\bar{x})+\partial^2\varphi(x^k)(-x^k+\bar{x}).
$$
 {Since $-\nabla\varphi(x^k)-\mu_k d^k\in\partial^2\varphi(x^k)(d^k)$, there exists $v^k\in\partial^2\varphi(x^k)(-x^k+\bar{x})$ such that
$$
-\nabla\varphi(x^k)-\mu_kd^k-v^k\in\partial^2\varphi(x^k)(x^k+d^k-\bar{x}).
$$
For large $k\in\N$ with $x^k\in\mathbb{B}_\delta(\ox)$, applying \eqref{uniformPSD} yields
\begin{equation}\label{directiondk}
\langle\nabla\varphi(x^k),-d^k\rangle\ge\big(\kappa^{-1}+\mu_k\big)\|d^k\|^2 \ge\kappa^{-1}\|d^k\|^2,
\end{equation}
$$
\langle-\nabla\varphi(x^k)-\mu_kd^k-v^k, x^k+d^k-\bar{x} \rangle \geq \frac{1}{\kappa}\|x^k+d^k-\bar{x}\|^2.
$$
Using again the Cauchy-Schwarz inequality together with $\nabla\varphi(\ox)=0$ ensures that
\begin{equation} \label{super2}
\|d^k\|\le\kappa\|\nabla\varphi(x^k)\|=\kappa\|\nabla\varphi(x^k)-\nabla\varphi(\bar{x})\|\le\kappa\ell\|x^k-\bar{x}\|,
\end{equation}
\begin{equation} \label{super1}
\|x^k+d^k-\bar{x}\|\le\kappa \|\nabla\varphi(x^k)+v^k+\mu_k d^k\|\le\kappa\left(\|\nabla\varphi(x^k)+v^k\|+\mu_k\|d^k\|\right).
\end{equation}
Furthermore, it follows from the Lipschitz continuity of $\nabla\varphi$ on $\mathbb{B}_\delta(\bar{x})$ and from $\nabla\varphi(\bar{x})=0$ that
\begin{equation}\label{muk}
\mu_k = c\|\nabla\varphi(x^k)\|  = c\|\nabla\varphi(x^k)-\nabla\varphi(\bar{x}) \|\le\ell\|x^k-\bar{x}\|.
\end{equation}
Combining \eqref{super2}, \eqref{super1}, \eqref{muk} tells us that
\begin{equation}\label{estiLM}
\|x^k+d^k-\ox\|\leq\kappa\|\nabla\varphi(x^k)+v^k\|+\kappa^2\ell^2\|x^k-\ox\|^2.
\end{equation}
Using now the semismooth$^*$ property of the gradient mapping $\nabla\varphi$ at $\bar{x}$ together with the stationarity condition $\nabla\varphi(\bar{x})=0$ implies by \cite[Lemma~5.5]{BorisKhanhPhat} that
\begin{equation}\label{semi22}
\|\nabla\varphi(x^k)+v^k\|=\|\nabla\varphi(x^k)-\nabla\varphi(\bar{x})+v^k\|=o(\|x^k-\bar{x}\|).
\end{equation}
It follows from \eqref{estiLM} and \eqref{semi22} that $\|x^k+d^k-\bar{x}\| = o(\|x^k-\bar{x}\|)$ as $k\to\infty$, which verifies this claim}.\\[1ex]
{\bf Claim~6:} {\em We have $\tau_k=1$ for all large $k\in\N$ provided that $\nabla\varphi$ is semismooth$^*$ at $\bar{x}$ and that either condition {\rm(a)}, or condition {\rm(b)} of the theorem holds}. To proceed, it suffices to verify the estimate in \eqref{backtrackinghold} under both conditions (a) and (b). {If (a) is satisfied}, then $\nabla\varphi$ is semismooth at
$\bar{x}$. Then this estimate and the assertion of the claim follows directly by \eqref{directiondk} and  {Proposition \ref{acceptancesm}}. Assuming now the condition in (b) and using Claim~5, we easily see that $\{d^k\}$ converges to $0$ and that $x^k+d^k\to\bar{x}$ as $k \to\infty$. Employing \eqref{second-growth2} justifies \eqref{growthdk}. Then \eqref{backtrackinghold} follows from Lemma~\ref{Maratos} and thus completes the verification of the claim.\\[1ex]
{\bf Claim~7:} {\em The Q-superlinear convergence holds in both cases ${\rm(a)}$ and ${\rm(b)}$ of ${\rm(iii)}$.} Indeed, we get from Claim~6 that $\tau_k=1$ for all $k$ sufficiently large. It follows from   Claim~5 \color{black}
that
$$
\|x^{k+1}-\bar{x}\|=\|x^k+\tau_k d^k-\bar{x}\|=\|x^k +d^k-\bar{x}\|=o(\|x^k-\bar{x}\|)\;\text{ as }\;k\to\infty,
$$
which justifies the Q-superlinear convergence of $\{x^k\}$ in both cases. The Q-superlinear convergence of $\{\varphi(x^k) \}$ follows immediately from \eqref{second-growth12} and \eqref{Lipschitzineq2}, while the Q-superlinear convergence of $\{\nabla \varphi(x^k) \}$ is a consequence of \eqref{strongineq2} and \eqref{muk}. This completes the proof of the theorem.
\end{proof}\vspace*{-0.15in}

 \begin{Remark}[\bf comparison with related globalized Newton-type algorithms] {\rm Observe the following:

{\bf(i)} In contrast to the generalized damped Newton algorithm (Algorithm \ref{LS}), the solvability of subproblems and the behavior of accumulation points in Algorithm \ref{LS2} are guaranteed merely under the {\em positive-semidefiniteness} of $\partial^2\varphi(x)$ in Theorem \ref{limitingpoint}. To achieve the convergence of the iterative sequence $\{x^k\}$ in Algorithm~\ref{LS2}, we only need the metric regularity of $\nabla\varphi$ around the accumulation point $\ox$ instead of the positive-definiteness of $\partial^2\varphi(x)$ for all $x \in \Omega$, which is the key assumption of the convergence in Algorithm~\ref{LS}.  Note also that this is just a sufficient condition to ensure the convergence of Algorithm~\ref{LS2}. The crucial open question we will pursue in our future research is whether it is possible to replace the metric regularity of $\nabla\varphi$ in Theorem~\ref{linearcon2} by a weaker assumption. One of the natural assumptions of this type is that $\|\nabla\varphi(x)\|$ provides a local error bound near the accumulation point, which was investigated in, e.g., \cite{DYF,LFQY,YF} in different settings.
 
{\bf(ii)} The generalized regularized Newton methods via the Bouligand Jacobian can be found in Pang and Qi \cite{PQ95} and in the book \cite[Section~8.3.3]{JPang}. Their method aims to solve the optimization problem 
\begin{equation}
\min\;\varphi(x)\;\text{ subject to }\;x\in P,
\end{equation}
where $P$ is a nonempty polyhedron in $\R^n$, $\varphi$ is a convex $\mathcal{C}^{1,1}$ function defined on an open convex set $\Omega \subset\R^n$ containing $P$.  
In the case of unconstrained minimization \eqref{opproblem} with $P=\R^n$, the generalized regularized Newton method via the Bouligand Jacobian requires to find direction $d^k$ as solutions to the equation
$$
 -\nabla\varphi(x^k)= (A^k+\epsilon_k I) d^k,\quad \text{where }\; A^k \in \partial_B \nabla\varphi(x^k), \; \text{and } \; \epsilon_k> 0.
$$ 
The key assumption to guarantee the convergence of their methods is that all the matrices in $\partial_B\nabla\varphi(\ox)$ are nonsingular, where $\ox$ is the accumulation point of the iterative sequence $\{x^k\}$ generated by their method; see, e.g., \cite[Theorem~8.3.19]{JPang}. This assumption is weaker than our assumption that $\nabla\varphi$ is metrically regular around $(\ox,0)$ due to the coderivative criterion \eqref{Morcri} and the inclusion
$$
\partial_B\nabla\varphi(\ox)w \subset\partial^2\varphi(\ox)(w) \quad \text{for all }\; w \in \R^n. 
$$
Observe to this end that the metric regularity property is defined for arbitrary set-valued mappings, and it is used for subgradient ones in this paper to guarantee the convergence; see Section~\ref{sec:dampnon}. Note also the calculus rules developed for the $\partial_B \nabla\varphi$ are more limited in comparison with full calculus available for $\partial^2\varphi$.}
\end{Remark}\vspace*{-0.2in}

\color{black}\section{Coderivative-Based Newton Methods in Composite Optimization}\label{sec:dampnon}\vspace*{-0.05in}

In this section, we consider a broad and highly important class of optimization problems given by
\begin{eqnarray}\label{QP0}
\text{minimize }\;\varphi(x):=f(x)+g(x),\quad x\in\R^n,
\end{eqnarray}
where $f\colon\R^n\to\R$ is a convex and smooth function, while the regularizer $g\colon\R^n\to\oR$ is a convex and extended-real-valued one. This class is known as problems of {\em convex composite optimization}.

Problems written in format \eqref{QP0} frequently arise in many applied areas including machine learning, compressed sensing, image processing, etc. Since the regularizer $g$ is generally extended-real-valued, the unconstrained format \eqref{comp} encompasses problems of {\em constrained optimization}. If, in particular, $g$ is the indicator function of a closed and convex set, then \eqref{comp} becomes a constrained optimization problems studied, e.g., in the book \cite{nw} with numerous applications.

One of the {most well-recognized and applied algorithms} to solve problems \eqref{comp} is the forward-backward splitting (FBS), or proximal splitting, method \cite{cp,lm}. Since this method is of first order, its rate of convergence is at most linear. Another approach to solve \eqref{comp} is to use second-order methods such as proximal Newton methods, proximal quasi-Newton methods, etc.; see, e.g., \cite{bf,lss,myzz}. Although the latter approach has several benefits over first-order methods (as fast convergence and high accuracy), a severe limitation of these methods is the cost of solving subproblems.\vspace*{0.03in}

To develop here new globally convergent Newton methods to solve convex composite optimization problems of type \eqref{QP0}, we first recall the classical notions of convex and variational analysis; see, e.g., \cite{Rockafellar98}. Given an extended-real-valued, proper, l.s.c.\ function $\varphi\colon\R^n\to\oR$ and a number $\gamma>0$, the {\em Moreau envelope} $e_\gamma\varphi$ and the {\em proximal mapping} $\textit{\rm Prox}_{\gamma\varphi}$ are defined by, respectively,
\begin{equation}\label{Moreau}
e_\gamma\varphi(x):=\inf_{y \in \R^n}\left\{\varphi(y)+\frac{1}{2\gamma}\|y-x\|^2\right\},
\end{equation}
\begin{equation}\label{Prox}
\textit{\rm Prox}_{\gamma\varphi}(x):=\underset{y\in \R^n}{\operatorname{argmin}}\left\{\varphi(y)+\frac{1}{2\gamma}\|y-x\|^2\right\}.
\end{equation}
If $\lambda=1$, we use the notation $e_\varphi(x)$ and $\text{\rm Prox}_{\varphi}(x)$ in \eqref{Moreau} and \eqref{Prox}, respectively. These notions have been well investigated in variational analysis and optimization as efficient tools of regularization and approximation of nonsmooth functions. Given a closed set $\emp\ne C\subset\R^n$, the \textit{orthogonal projection mapping} $P_C:\R^n\rightrightarrows\R^n$ is
$$
P_C(x):=\operatorname{argmin}_{y\in C}\|y-x\|\quad\text{for all }\; x\in \R^n.
$$
It is clear that if $g:\R^n\to (-\infty,\infty]$ be defined by $g(x):=\delta_C(x)$, we have
$$
\textit{\rm Prox}_{\gamma g}(x)=P_C(x)\quad\text{for all }\; x\in\R^n,\;\gamma >0. 
$$
More recently, the following extended notion, known now as the {\em forward-backward envelope}, has been introduced by Patrinos and Bemporad \cite{pb} for problems of convex composite optimization.

\begin{Definition}[\bf forward-backward envelope]\label{def:FBE} Let $\varphi=f+g$ be as in \eqref{QP0}, and let $\gamma>0$. The \textsc{forward-backward envelope} $($FBE$)$ of $\varphi$ with parameter $\gamma$ is
\begin{equation}\label{FBE}
\varphi_\gamma(x):=\inf_{y\in\R^n} \left\{f(x)+\langle\nabla f(x),y -x\rangle+g(y)+\frac{1}{2\gamma}\|y-x\|^2\right\},
\end{equation}
{which, by construction \eqref{Moreau} of the Moreau envelope, is represented by}
\begin{equation}\label{FBE2}
\varphi_\gamma(x)=f(x)-\frac{\gamma}{2}\|\nabla f(x)\|^2+ e_\gamma g\big(x-\gamma\nabla f(x)\big).
\end{equation}
\end{Definition}
The FBE has already been used for developing some efficient algorithms to solve nonsmooth optimization problems; see, e.g., \cite{pb,stp,stp2} with further references therein. The following results taken from \cite{pb,stp} list those properties of the forward-backward envelope for convex composite extended-real-valued functions that are needed to derive the main results of this section.

\begin{Proposition}[\bf basic properties of FBE] \label{diffFBE} Let $\varphi=f+g$ be as in \eqref{QP0}, and let $\gamma>0$. Suppose that $f$ is ${\cal C}^2$-smooth on $\R^n$, and that $\nabla f$ is Lipschitz continuous on $\R^n$ with modulus $\ell>0$. Then we have:

{\bf(i)} The FBE $\varphi_\gamma$ of $\varphi$ is ${\cal C}^1$-smooth on $\R^n$ with the gradient
\begin{equation}\label{gradFBE}
\nabla\varphi_\gamma(x)=\gamma^{-1}\big(I-\gamma \nabla^2f(x)\big)\big(x-\text{\rm Prox}_{\gamma g}(x-\gamma \nabla f(x))\big),\quad x\in\R^n.
\end{equation}
Moreover, the set of optimal solutions to \eqref{QP0} agrees with the stationary points of $\ph_\gg$ by
$$
\text{\rm argmin}\,\varphi:=\text{\rm zer}\,\nabla \varphi_\gamma=\big\{x\in\R^n\;\big|\;\nabla\varphi_\gamma(x)=0\big\}\quad\text{for all }\;\gamma\in(0,1/\ell).
$$

{\bf(ii)} Let $f(x):=\frac{1}{2}\langle Ax,x\rangle+\langle b,x\rangle+\alpha$, where $A\in\R^{n\times n}$ is a
positive-semidefinite symmetric matrix, $b\in\R^n$, and $\alpha\in\R$. Define the numbers
\begin{equation*}
L:= 2\left(1-\gamma \lambda_{\text{\rm min}(A)}\right)/\gamma\;\mbox{ and }\;K:=\text{\rm min}\big\{(1-\gamma\lambda_{\text{\rm min}(A)})\lambda_{\text{\rm min}(A)}, (1-\gamma\lambda_{\text{\rm max}(A)})\lambda_{\text{\rm max}(A)}\big\}.
\end{equation*}
Then for all $\gamma\in(0,1/\ell)$, the FBE $\varphi_\gamma$ is convex and its gradient $\nabla \varphi_\gamma$ is globally Lipschitzian on $\R^n$ with modulus $L$. If $A$ is positive-definite, then $\varphi_\gamma$ is strongly convex with modulus $K$.
\end{Proposition}

It follows from Proposition~\ref{diffFBE} that using the forward-backward envelope \eqref{FBE} makes it possible to pass from  {the nonsmooth composite optimization problem} \eqref{QP0} to the unconstrained one:
\begin{equation}\label{FBEop}
\min \quad \varphi_\gamma(x)\quad \text{subject to }\;x\in\R^n
\end{equation}
with a smooth cost function. Thanks to the explicit calculations of $\varphi_\gamma$ in \eqref{FBE2} and its gradient \eqref{gradFBE}, we can extend Algorithm~\ref{LS} and Algorithm~\ref{LS2} to cover problem \eqref{QP0} via passing to \eqref{FBEop}. The implementation of this procedure requires revealing appropriate assumptions on $\varphi$ in \eqref{QP0}, which ensure the fulfillment of those for $\ph_\gg$ and thus allow us to apply the results of Sections~\ref{sec:dampedC11}, \ref{sec:modifiedC11} to the class of nondifferentiable convex problems \eqref{FBEop}.

Note that \eqref{FBEop} is not generally a problem of ${\cal C}^{1,1}$ optimization, since Proposition~\ref{diffFBE}(i) does not ensure the Lipschitz continuity of $\nabla\ph_\gg$. The latter property is guaranteed by Proposition~\ref{diffFBE}(ii) when $f$ is a quadratic function and thus problem \eqref{QP0} is written as
\begin{eqnarray}\label{QP}
\text{minimize }\;\varphi(x):=\frac{1}{2}\langle Ax,x\rangle+\langle b,x\rangle+\alpha+g(x),\quad x\in\R^n,
\end{eqnarray}
where $A\in\R^{n\times n}$ is a {\em positive-semidefinite} symmetric matrix, $b\in\R^n$, and $\alpha\in\R$. From now on, this is our standing framework for the rest of the section.

Let us highlight that problems of type \eqref{QP} are important for their own sake, while they also arise frequently as {\em subproblems} for various efficient numerical algorithms including {\em sequential quadratic programming methods} (SQP) \cite{Bonnans,Solo14}, {\em augmented Lagrangian methods} \cite{He69,lsk,Po69}, {\em proximal Newton methods} \cite{lss,myzz}, etc. Observe furthermore that optimization problems of this type often appear in practical models related, e.g., to machine learning and statistics. In particular,  {\em Lasso problems} considered in Section~\ref{Lassosec} can be written in form \eqref{QP}. Moreover, there are some other important classes of problems that are modeled as \eqref{QP}. They include problems in {\em support vector machine} \cite{HCL}, {\em convex clustering} \cite{BMPS,she}, {\em constrained quadratic optimization} \cite{nw}, etc.\vspace*{0.03in}

Now we start the procedure of designing and justifying globally convergent generalized Newton algorithms to solve the convex composite problem \eqref{QP} by applying the corresponding results for the ${\cal C}^{1,1}$ optimization problem \eqref{FBEop} obtained in Sections~\ref{sec:dampedC11} and \ref{sec:modifiedC11}. The first step is  {to express} the generalized Hessian of  {the FBE} $\ph_\gg$ from \eqref{FBEop} in terms of the given data of \eqref{QP}. 

\begin{Proposition}[\bf calculating the generalized Hessian of FBE] \label{calculatepsi} Let $\varphi=f+g$ be as in \eqref{QP}, and let $\gamma >0$ be such that  {$B:=I -\gamma A$ is positive-definite}.  Then we have the calculation formula
\begin{equation}
\bar{z} \in \partial^2\varphi_\gamma (\bar{x})(w) \iff B^{-1} \bar{z} - Aw \in \partial^2 g \left(\text{\rm Prox}_{\gamma g}(\bar{u}),\frac{1}{\gamma}\big(\bar{u} -\text{\rm Prox}_{\gamma g}(\bar{u})\big) \right)\big(w-\gamma B^{-1}\bar{z}\big)
\end{equation}
for any $\bar{x} \in \R^n$, $w \in \R^n$, and $\bar{u}:= \bar{x}- \gamma(A\bar{x}+b)$.
\end{Proposition}
\begin{proof} Fix $x,w$, and $\ou$ as above and define the function $h: \R^n\to\overline{\R}$ by
$$
h(x):= e_{\gamma}g\big(x-\gamma(Ax+b)\big) \quad \text{for all }\; x \in \R^n.
$$
It is clear that $h$ is {continuously} differentiable with $\nabla h(\bar{x}) = (I-\gamma A)^* \nabla e_\gamma g (\bar{u}) = B\nabla e_\gamma g(\bar{u})$. It follows from the definition of FBE \eqref{FBE2} and the second-order sum rule from \cite[Proposition~1.121]{Mordukhovich06}  that
\begin{equation}\label{cal1}
\partial^2\varphi_\gamma(\bar{x})(w) = (A-\gamma A^*A)w + \partial^2 h(\bar{x})(w) = BAw + \partial^2 h(\bar{x})(w).
\end{equation}
By using the second-order chain rule from \cite[Theorem~1.127]{Mordukhovich06}, we have
\begin{equation}\label{cal2}
\partial^2h(\bar{x})(w)  = B\partial^2 e_\gamma g(\bar{u})(Bw).
\end{equation}
Combining \eqref{cal1} and \eqref{cal2} gives us the relationship
$$
\partial^2\varphi_\gamma(\bar{x})(w) = BAw + B\partial^2 e_\gamma g(\bar{u})(Bw),
$$
which in turn yields the equivalencies
$$
\bar{z} \in \partial^2\varphi_\gamma(\bar{x})(w) \iff \bar{z} - BAw \in B\partial^2e_\gamma g(\bar{u})(Bw) \iff B^{-1}\bar{z} - Aw \in \partial^2e_\gamma g(\bar{u})(Bw).
$$
Employing finally \cite[Lemma 6.4]{BorisKhanhPhat}, we arrive at the inclusion
$$
B^{-1}\bar{z} - Aw \in \partial^2 g\left(\text{\rm Prox}_{\gamma g}(\bar{u}), \frac{1}{\gamma}\big(\bar{u}-\text{\rm Prox}_{\gamma g}(\bar{u})\big)\right)(Bw -\gamma B^{-1}\bar{z}+\gamma Aw),
$$
which completes the proof of the proposition due to $Bw + \gamma Aw = w$.
\end{proof}

The next proposition shows that the metric regularity and tilt
stability of the original objective $\ph$ in \eqref{QP} is
equivalent to the corresponding properties of its FBE $\ph_\gamma$
in \eqref{FBEop}. Moreover, we get a useful estimate of the inverse
mapping of $\partial^2\ph_\gamma$ in terms of the given data of
\eqref{QP}.

\begin{Proposition}[\bf metric regularity and tilt-stability of FBE] \label{metrictilt}  Let $\varphi=f+g$ be as in \eqref{QP}, and let $\gamma>0$ be such that  {$B:=I -\gamma A$ is positive-definite}. Then for any $\ox\in\R^n$ satisfying $0\in\partial\varphi(\bar{x})$ we have:

{\bf(i)} $\|\partial^2\varphi_\gamma(\bar{x})^{-1}\|\le\|\partial^2\varphi(\bar{x},0)^{-1}\|+\gamma\|B^{-1}\|$.

{\bf(ii)} $\partial\varphi$ is metrically regular around $(\bar{x},0)$  if and only if $\nabla\varphi_\gamma$ is metrically regular
around $\ox$.

{\bf(iii)} $\bar{x}$ is a tilt-stable local minimizer of $\varphi$ if and only if $\bar{x}$ is a tilt-stable local minimizer of
$\varphi_\gamma$.
\end{Proposition}
\begin{proof} It follows from Proposition~\ref{calculatepsi} that
\begin{equation}\label{equi}
z \in \partial^2\varphi_\gamma(\bar{x})(w) \iff 0 \in Aw + \partial^2 g\left(\text{\rm Prox}_{\gamma g}(\bar{u}),\frac{1}{\gamma}\big(\bar{u} -\text{\rm Prox}_{\gamma g}(\bar{u})\big) \right)\big(w-\gamma B^{-1}z\big)
\end{equation}
with $\bar{u}$ defined therein. The convexity of $\varphi$ ensures that $\bar{x}$ is an optimal solution to \eqref{QP}, and thus $\bar{x} - \text{\rm Prox}_{\gamma g}(\bar{u}) =0$ by \cite[Theorem~27.2]{Bauschke}. Therefore, \eqref{equi} is equivalent to
\begin{equation}\label{equa1}
z\in Aw+\partial^2 g(\bar{x},-A\bar{x}-b)(w-\gamma B^{-1}z)=A(w-\gamma B^{-1}z)+\partial^2 g(\bar{x},-A\bar{x}-b)(w-\gamma B^{-1}z)+\gamma AB^{-1}z.
\end{equation}
The second-order subdifferential sum rule from \cite[Proposition~1.121]{Mordukhovich06} yields
\begin{eqnarray}\label{equa2}
A(w-\gamma B^{-1}z) + \partial^2 g(\bar{x},-A\bar{x}-b)(w-\gamma B^{-1}z)&=&\partial^2(f+ g)(\bar{x},0)(w-\gamma B^{-1}z)\nonumber \\
&=&\partial^2\varphi(\bar{x},0)(w-\gamma B^{-1}z).
\end{eqnarray}
Combining \eqref{equi}, \eqref{equa1}, and \eqref{equa2} gives us the equivalence
\begin{equation}\label{equiFBE}
z\in\partial^2\varphi_\gamma(\bar{x})(w) \iff  z \in \partial^2\varphi(\bar{x},0)(w-\gamma B^{-1}z),
\end{equation}
which verifies (i). It follows from the coderivative criterion \eqref{Morcri} and the equivalence \eqref{equiFBE} that $\partial\varphi$ is metrically regular around $(\bar{x},0)$   if and only if $\nabla\varphi_\gamma$ is metrically regular around $\ox$, which justifies assertion (ii). Finally, \cite[Proposition~4.5]{dmn} tells us that a stationary point $\bar{x}$ is a tilt-stable local minimizer of an l.s.c.\ convex function if and only if its subgradient mapping is metrically regular around $(\bar{x},0)$. Using this observation together with (ii), we obtain (iii) and complete the proof of the proposition.
\end{proof}

Now we recall some other notions of variational analysis that are used to establish the superlinear convergence of both algorithms developed in this section to solve problem \eqref{QP}. These notions, introduced by Rockafellar, are taken from the book \cite{Rockafellar98}. A set-valued mapping $S:\R^n\rightrightarrows \R^m$ is {\em proto-differentiable} at $(\bar{x},\bar{y})\in\gph S$ if for any $\bar{w}\in\R^n$, $ \bar{z}\in \underset{t\downarrow 0 \atop w\to \bar{w}}{\Limsup}\; (S(\bar{x}+tw)-\bar{y})/t $, and $t_k\downarrow 0$ there exist $w_k \to \bar{w}$ and $z_k \to \bar{z}$ such that $z_k \in (S(\bar{x}+t_kw_k)-\bar{y})/t_k$ whenever $k \in \N$.
Given $\ph\colon\R^n\to\oR$ with $\ox\in\dom\ph$, consider the family of second-order finite differences
\begin{equation*}
\Delta^2_\tau\varphi(\bar{x},v)(u):=\frac{\varphi(\bar{x}+\tau u)-\varphi(\bar{x})-\tau\langle v,u\rangle}{\frac{1}{2}\tau^2}
\end{equation*}
and define the {\em second subderivative} of $\varphi$ at $\ox$ for $v\in\R^n$ and $w\in\R^n$ by
\begin{equation*}
d^2\varphi(\ox,v)(w):=\liminf_{\tau\downarrow 0\atop u\to w}\Delta^2_\tau\varphi(\ox,v)(u).
\end{equation*}
Then $\ph$ is said to be {\em twice epi-differentiable} at $\ox$ for $v$ if for every $w\in\R^n$ and every choice of $\tau_k\downarrow 0$ there exists a sequence $w^k\to w$ such that
\begin{equation*}
\frac{\varphi(\ox+\tau_k w^k)-\varphi(\ox)-\tau_k\langle v,w^k\rangle}{\frac{1}{2}\tau_k^2}\to d^2\varphi(\ox,v)(w)\;\mbox{ as }\;k\to\infty.
\end{equation*}
Twice epi-differentiability has been recognized as an important concept of second-order variational analysis with numerous applications to optimization; see the aforementioned monograph by Rockafellar and Wets and the recent papers \cite{mms,mms1,ms} developing a systematic approach to verify epi-differentiability via {\em parabolic regularity}, which is a major second-order property of sets and functions.\vspace*{0.03in}

The next proposition expresses the properties of the FBE $\ph_\gg$ in \eqref{FBEop}, which are needed for the superlinear convergence of our algorithms, in terms of the given data of \eqref{QP}. {Recall that the sign
`$>$' before a matrix indicates the matrix positive-definiteness.}

\begin{Proposition}[\bf semismoothness$^*$ and directional differentiability of FBE derivatives]\label{semidirect}  Let $\varphi=f+g$ be as in \eqref{QP}, and let $\gamma >0$ be such that {$B:=I -\gamma A$ is positive-definite}. Then for any $\bar{x}\in\R^n$ satisfying the stationary condition $0 \in \partial\varphi(\bar{x})$ the following assertions hold:

{\bf(i)} $\nabla\varphi_\gamma$ is semismooth$^*$ at $\bar{x}$ if $\partial g$ is semismooth$^*$ at $(\bar{x},\ov)$, where
$\ov:=-A\bar{x}-b$.

{\bf(ii)} $\nabla\varphi_\gamma$ is directionally differentiable at $\bar{x}$ if $g$ is twice epi-differentiable at $\bar{x}$ for
$\bar{v}$.
\end{Proposition}
\begin{proof}
Denote $h_\gamma(x):=\text{\rm Prox}_{\gamma g}(x-\gamma(Ax+b))$ for all $x\in \R^n$ and get by Proposition~\ref{diffFBE} that
\begin{equation}\label{FBEh}
 {\nabla\varphi_\gamma(x)=\gamma^{-1}(I-\gamma A)\big(x-h_\gamma(x)\big)=\gamma^{-1} Bx-\gamma Bh_\gamma^{-1}(x),\quad x \in \R^n.}
\end{equation}
Since the stationary point $\ox$ is an optimal solution to \eqref{QP} due the convexity of $\ph$, we have that $\bar{x}= h_\gamma(\bar{x})$ by \cite[Theorem~27.2]{Bauschke}. Since $\partial g$ is semismooth$^*$ at $(\bar{x},\ov)$, we have that $\text{\rm Prox}_{\gamma g}$ is semismooth$^*$ at $\bar{x}-\gamma(A\bar{x}+b)$ by using \cite[Proposition~6]{fgh}. It follows from Lemma~\ref{semismoothcomposite} in the Appendix that $h_\gamma$ is semismooth$^*$ at $\bar{x}$. Employing now \eqref{FBEh} and \cite[Proposition~3.6]{Helmut} tells us that  $\nabla\varphi_\gamma$ is semismooth$^*$ at $\bar{x}$. This verifies assertion (i).

To proceed with the proof of (ii), observe by \cite[Theorem~13.40]{Rockafellar98} that the twice epi-differentiability of $g$ at $\bar{x}$ for $\bar{v}$ amounts to saying that the subgradient mapping $\partial g$ is proto-differentiable  at $(\bar{x},\bar{v})$.
Using \cite[Corollary~8]{fgh}, we conclude that $\text{\rm Prox}_{\gamma g}$ is directionally differentiable at $\bar{x}-\gamma(A\bar{x}+b)$, which yields in turn the directional differentiability of $h_\gamma$ at $\bar{x}$. Thus the mapping $\nabla\varphi_\gamma$ is directionally differentiable at $\bar{x}$ due to \eqref{FBEh}. This verifies (ii) and completes the proof of the proposition.
\end{proof}

Now we are ready to describe and then justify the proposed globally coderivative-based damped Newton method for solving the convex composite optimization problem \eqref{QP}.

\begin{algorithm}[H]
\caption{Coderivative-based damped Newton algorithm for convex composite optimization}\label{LSQP}
\begin{algorithmic}[1]
\Require {$x^0 \in \R^n$, $\gamma>0$ such that $B:=I-\gamma A > 0$, $\sigma\in\left(0,\frac{1}{2}\right)$, $\beta\in(0,1)$, and  $\varphi_\gamma$ as in \eqref{FBE}} \For{\texttt{$k=0,1,\ldots$}}
\State\text{If $\nabla\varphi_\gamma(x^k)=0$, stop. Otherwise set $
u^k:= x^k - \gamma(Ax^k+b), \; v^k:=\text{\rm Prox}_{\gamma g}(u^k)$}
\State \label{prox-dir} \text{Find $d^k\in\R^n$ st
$-\frac{1}{\gamma}(x^k-v^k)-Ad^k \in \partial^2 g\left(v^k, \frac{1}{\gamma}(u^k-v^k) \right)(x^k-v^k + d^k)$}
\State  Set $\tau_k = 1$
\While{$\varphi_\gamma(x^k+\tau_kd^k)>\varphi_\gamma(x^k)+\sigma\tau_k\langle\nabla\varphi_\gamma(x^k),d^k\rangle$}
\State \text{set $\tau_k:= \beta \tau_k$}
\EndWhile
\State \text{Set $x^{k+1}:=x^k+\tau_k d^k$}
\EndFor
\end{algorithmic}
\end{algorithm}
Explicit expressions for the sequences $\{v^k\}$ and $\{d^k\}$ in Algorithm \ref{LSQP} depend on given structures of the regularizers $g$, which are efficiently specified in applied models of machine learning and statistics; see, e.g., Section~\ref{Lassosec}.
The next theorem provides explicit sufficient conditions to run Algorithm~\ref{LSQP} for solving the class of convex composite optimization problems \eqref{QP}.

\begin{Theorem}[\bf global convergence of coderivative-based damped Newton algorithm in convex composite optimization]\label{solvingQP} Consider problem \eqref{QP}, where the matrix $A$ is positive-definite. Then Algorithm~{\rm\ref{LSQP}} either stops after finitely many iterations, or produces a sequence $\{x^k\}$ such that it globally R-linearly converges to $\bar{x}$, which is the unique solution to \eqref{QP} being a tilt-stable local minimizer of $\varphi$ with modulus $\kappa:=1/\lambda_{\text{\rm min}(A)}$. Furthermore, the convergence rate of $\{x^k\}$ is at least Q-superlinear if the subgradient mapping $\partial g$ is semismooth$^*$ at $(\bar{x},\bar{v})$, where $\bar{v}:=-A\bar{x}-b$, and if either one of two following conditions is satisfied:

{\bf(i)} $\sigma\in\left(0,1/(2L K)\right)$, where
$L:= 2\left(1-\gamma\lambda_{\text{\rm min}(A)}\right)/\gamma$ and $K:=\kappa+\gamma\|B^{-1}\|$.

{\bf(ii)} $g$ is twice epi-differentiable at $\bar{x}$ for $\bar{v}$.
\end{Theorem}
\begin{proof} We deduce from Propositions~\ref{diffFBE}(ii) and \ref{calculatepsi} that solving
the convex composite optimization problem \eqref{QP} by Algorithm~\ref{LSQP} reduces to solving the ${\cal C}^{1,1}$ optimization problem \eqref{FBEop} by using Algorithm~\ref{LS}. The rest of the proof is split into the following two claims.\\[1ex]
{\bf Claim~1:} {\em Algorithm~{\rm\ref{LSQP}} either stops after finitely many iterations, or produces a sequence sequence $\{x^k\}$ that globally R-linearly converges to the unique solution $\bar{x}$ of \eqref{QP}, which is a tilt-stable local minimizer of $\varphi$.} Indeed, we get from Proposition~\ref{diffFBE}(ii) that $\varphi_\gamma$ is a strongly convex function, and its gradient is globally Lipschitz continuous with modulus $L$. It follows from\cite[Theorem~5.1]{ChieuChuongYaoYen} that $\partial^2\varphi_\gamma(x)$ is positive-definite for all $x \in \R^n$. Then Theorem~\ref{globalconver} tells us that Algorithm~{\rm\ref{LSQP}} either stops after finitely many iterations, or produces a sequence $\{x^k\}$ that globally $R$-linearly converges to $\bar{x}$, which is a stationary point of $\varphi_\gamma$. Employing again Proposition~\ref{diffFBE}(ii) confirms that $\bar{x}$ is an optimal solution to \eqref{QP}. Furthermore, the strong convexity of $\ph$ with modulus $\lambda_{\text{\rm min}(A)}$ and Lemma~\ref{strongtilt} from the Appendix yield the uniqueness and tilt stability conclusions for $\ox$. \\[1ex]
{\bf Claim~2:} {\em The Q-superlinear convergence of $x^k\to\ox$ holds under the assumptions of the theorem.} The imposed semismooth$^*$ property of $\partial g$ at $(\ox,\ov)$ ensures the fulfillment of this property for $\nabla\ph_\gg$ at $\ox$ by Lemma~\ref{semidirect}. Assume now that condition (i) of the theorem is satisfied. Then we get from Claim~1 that $L$ is a Lipschitz constant of $\nabla\varphi_\gamma$ around $\ox$. As follows from \cite[Proposition~4.5]{dmn}, the modulus of tilt-stability of the l.s.c.\ convex function $\varphi$  under consideration at $\ox$ is the same as the modulus of metric regularity of $\varphi$ around this point. Combining the latter with the statement of Proposition~\ref{metrictilt}(i) and the precise calculation in \eqref{Morcri2} of the exact bound of metric regularity, we conclude that $\ox$ is tilt-stable local minimizer of $\varphi_\gamma$ with modulus $K$. Thus the claimed assertion on the superlinear convergence in this case follows directly from Theorem~\ref{superlinearLS}. Assuming finally by (ii) that $g$ is twice epi-differentiable at $\ox$ for $\ov$, we deduce from Proposition~\ref{semidirect}(ii) that the mapping $\nabla\varphi_\gamma$ is directionally differentiable at $\ox$. Thus the claimed superlinear convergence of $\{x^k\}$ follows in this case from the corresponding statement of Theorem~\ref{superlinearLS}. This completes the proof of the theorem.
\end{proof}

Theorem~\ref{solvingQP} and the results of numerical experiments in Section~\ref{Lassosec} show that Algorithm~\ref{LSQP}, designed in terms of the computable data of \eqref{QP}, exhibits an excellent performance when $A$ is positive-definite, i.e., in the strongly convex setting of \eqref{QP}. Otherwise, this algorithm is not even well-defined. To relax this positive-definiteness/strong convex assumption, we now propose and justify a new coderivative-based algorithm of the  {regularized Newton} type, which is well-defined and globally convergent to solutions of \eqref{QP} for merely {\em positive-semidefinite} matrices $A$ with linear and superlinear convergence rates under some additional assumptions that include the {\em metric regularity} of the subgradient mapping $\partial\ph$. Observe that the latter assumption is {\em weaker} than  {the {\em strong convexity assumption} on $f$} in problems of composite optimization \eqref{QP0}, including those with quadratic functions $f$ as in \eqref{QP}. A simple class of functions $\ph$ illustrating this observation is given by $\ph=f+|x|$. In particular, for $f\equiv 0$ we clearly have that
\begin{equation*}
 {\partial^2\ph(0,0)(v)=\big\{w\in\R\;\big|\;(w,-v)\in\R\times\{0\}\big\},}
\end{equation*}
and thus $0\in\partial^2\ph(0,0)(v)\Longrightarrow v=0$, which tells us by \eqref{Morcri} that $\partial\ph$ is metrically regular around $(0,0)$.\vspace*{0.03in}

Here is the aforementioned algorithm, which is more complicated than Algorithm~\ref{LSQP}, while being applied for problems \eqref{QP} with positive-semidefinite matrices $A$. Note that the new algorithm {\em does not} require performing operations like computing inverse matrices that are expensive in large dimensions.

\begin{algorithm}[H]
\caption{{Coderivative-based regularized Newton algorithm for convex composite optimization}}\label{LSQP2}
\begin{algorithmic}[1]
\Require {$x^0 \in \R^n$, $\gamma>0$ such that $B:=I-\gamma A >0$, $\lambda>0$,  $\sigma\in\left(0,\frac{1}{2}\right)$, $\beta\in(0,1)$, and   $\varphi_\gamma$ as in \eqref{FBE}}
\For{\texttt{$k=0,1,\ldots$}}
\State\text{If $\nabla\varphi_\gamma(x^k)=0$, stop. Otherwise set $u^k:= x^k - \gamma(Ax^k+b),v^k:=\text{\rm Prox}_{\gamma g}(u^k),\mu_k:= \lambda\|\nabla\varphi_\gamma(x^k)\|$}
\State \label{prox-dir2} \text{Find $z^k\in\R^n$ st
$-\frac{1}{\gamma}(x^k-v^k)-(\mu_k I+AB)z^k\in\partial^2 g\left(v^k,\frac{1}{\gamma}(u^k-v^k) \right)\big(x^k-v^k +(B+\gg\mu_kI)z^k\big)$}
\State Set $d^k=B z^k$
\State  Set $\tau_k = 1$
\While{$\varphi_\gamma(x^k+\tau_kd^k)>\varphi_\gamma(x^k)+\sigma\tau_k\langle\nabla\varphi_\gamma(x^k),d^k\rangle$}
\State \text{set $\tau_k:= \beta \tau_k$}
\EndWhile
\State\text{Set $x^{k+1}:=x^k+\tau_k d^k$}
\EndFor
\end{algorithmic}
\end{algorithm}

The next theorem fully describes the well-posedness and performance of Algorithm~\ref{LSQP2}.

\begin{Theorem}[\bf {global convergence of coderivative-based  regularized Newton algorithm in convex composite optimization}]\label{solvingQP2} Consider problem \eqref{QP} of convex composite optimization, where the matrix $A$ is positive-semidefinite. Then we have the assertions:

{\bf(i)} Algorithm~{\rm\ref{LSQP2}} either stops after finitely many iterations, or produces a sequence $\{x^k\}$ for which all the {accumulation points} of this sequence are optimal solutions to \eqref{QP}.

{\bf(ii)} If in addition the subgradient mapping $\partial\varphi$ is metrically regular around $(\bar{x},0)$ with modulus $\kappa>0$, where $\bar{x}$ is {an accumulation point} of $\{x^k\}$, then the sequence $\{x^k\}$ globally R-linearly converges to $\bar{x}$, and $\bar{x}$ is a tilt-stable local minimizer of $\varphi$ with modulus $\kappa$.

{\bf(iii)} The rate of convergence of $\{x^k\}$ is at least Q-superlinear if the subgradient mapping $\partial g$ is semismooth$^*$ at $(\bar{x},\bar{v})$, where $\bar{v}:=-A\bar{x}-b$, and if one of two following conditions holds:

{\bf(a)} $\sigma\in\left(0,1/(2LK)\right)$, where $L:= 2\left(1-\gamma\lambda_{\text{\rm min}(A)}\right) /\gamma$ and $K:= \kappa+\gamma\|B^{-1}\|$.

{\bf(b)} $g$ is twice epi-differentiable at $\bar{x}$ for $\bar{v}$.
\end{Theorem}
\begin{proof}
Using Propositions~\ref{diffFBE}(ii) and \ref{calculatepsi}, we can reduce Algorithm~\ref{LSQP2} for solving the convex composite optimization problem \eqref{QP} to Algorithm~\ref{LS2} for solving the ${\cal C}^{1,1}$ optimization problem \eqref{FBEop}. Let us now proceed with the justification of this procedure by verifying each claim of the theorem.\\[1ex]
{\bf Claim~1:} {\em Assertion {\rm(i)} holds.} Indeed, we have from Proposition~\ref{diffFBE}(ii) that the gradient mapping $\nabla\varphi_\gamma$ is globally Lipschitz continuous with modulus $L$ and $\varphi_\gamma$ is a convex function; thus the generalized Hessian
$\partial^2\varphi_\gamma(x)$ is positive-semidefinite for all $x\in\R^n$ by \cite[Theorem~3.2]{ChieuChuongYaoYen}. Therefore, it follows from Theorem~\ref{limitingpoint} that Algorithm~{\rm\ref{LSQP2}} either stops after finitely many iterations, or produces a sequence $\{x^k\}$ whose
{accumulation points} are solutions to \eqref{FBEop}. This verifies assertion (i).\\[1ex]
{\bf Claim~2:} {\em Assertion {\rm(ii)} holds.} Under the assumptions made in (ii), it follows from the established relationships between problems \eqref{QP} and \eqref{FBEop} and the application of Theorem~\ref{linearcon2} to the latter that the sequence $\{x^k\}$ globally $R$-linearly converges to $\bar{x}$ as $k\to\infty$. Then we get by \cite[Proposition~4.5]{dmn} that the tilt-stability of $\varphi$ at $\bar{x}$ with modulus $\kappa$ follows from the metric regularity of $\partial\varphi$ and the convexity of $\varphi$, which therefore verifies (ii).
\\[1ex]
{\bf Claim~3:} {\em Assertion {\rm(iii)} holds.} We deduce from Proposition~\ref{semidirect}(i) that the  semismoothness$^*$ {of $\partial g$ at $(\ox,\ov)$} yields this property for $\nabla\varphi_\gamma$ at $\bar{x}$. Assuming first that condition (a) is satisfied, we deduce from Proposition~\ref{diffFBE}(ii) that $L$ is a Lipschitz constant of $\nabla\varphi_\gamma$ around $\ox$. It follows from \cite[Proposition~4.5]{dmn} that the modulus of tilt-stability of the l.s.c.\ convex function $\varphi$ at $\ox$ is equal to the modulus of metric regularity of $\varphi$ around this point. Combining the latter with Proposition~\ref{metrictilt}(i) and the calculation formula for the exact regularity bound in \eqref{Morcri2} tells us that $\ox$ is a tilt-stable local minimizer of $\varphi_\gamma$ with modulus $K$. Hence assertion (iii) in case (a) follows from Theorem~\ref{linearcon2}(a). Assuming now (b) implies by Proposition~\ref{semidirect}(ii) that $\nabla\varphi_\gamma$ is directionally differentiable at $\ox$. Applying finally Theorem~\ref{linearcon2}(b) to problem \eqref{FBEop} ensures  {the $Q$-superlinear  convergence} of sequence $x^k\to\bar{x}$ as $k \to\infty$ and therefore completes the proof of the theorem.
\end{proof}\vspace*{-0.13in}

\begin{Remark}[\bf second-order subdifferential computations for regularizers]{\rm One of the crucial steps in implementing Algorithms~\ref{LSQP} and \ref{LSQP2} is deriving explicit second-order subdifferential computations for the regularizers $g:\R^n\to\overline{\R}$ in the quadratic composite optimization problem \eqref{QP}. This has been accomplished in many publications, some of which we mention below for the reader's convenience and further applications. When $g$ is the indicator function of general {\em polyhedral sets}, explicit formulas of different types for $\partial^2g$ are derived in \cite{dr,hmn,hr,yy}. In the case of {\em moving polyhedra}, such computations are provided in \cite{chhm,nam10,qui14} with applications to optimal control of sweeping processes in \cite{chhm} among other publications. When $g$ is described  by {\em nonlinear inequality systems}, $\partial^2g$ is efficiently evaluated in \cite{hos}, and for various types of {\em maximum functions} the computations $\partial^2g$ are achieved in \cite{EH14,ms16}. Furthermore, in \cite{ms16}, the reader can find explicit computation formulas for $\partial^2g$ addressing the general class of extended-real-valued {\em convex piecewise linear} functions, while papers \cite{mr,mrs} contain second-order subdifferential computations for some different subclasses of piecewise linear-quadratic functions. Complete computations of $\partial^2g$ for the indicator function of the {\em Lorentz cone} in second-order cone programming are provided in \cite{mos,os}. Finally in this list, we mention the most involved second-order subdifferential computations accomplished in \cite{dsy} for general classes of the general class of nonpolyhedral systems including {\em semidefinite cone complementarity constraints}.} 
\end{Remark}\vspace*{-0.25in}

\section{Applications and Numerical Experiments}\label{Lassosec}\vspace*{-0.05in}
 
This section provides numerical experiments for our methods and their comparison with the classical semismooth Newton methods in $\mathcal{C}^{1,1}$ optimization. Furthermore, we also provide applications of the developed generalized Newton algorithms to problems of nonsmooth convex composite optimization and their comparison with some well-recognized methods of nonsmooth optimization in some classes of practical models.\vspace*{-0.1in} 

\subsection{Testing $\mathcal{C}^{1,1}$ Optimization Problems}\label{testC11experiment}

The first subsection here is devoted to comparing our approach with the well-recognized semismooth Newton method (SNM) to solve the {\em testing optimization problem}:
\begin{eqnarray}\label{C11testproblem}
\text{minimize }\;\varphi(x):= e_f(x) + \frac{1}{2}\|x\|^2 \quad\text{ subject to }\;x\in\R^n,
\end{eqnarray}
where $e_f(x)$ is the  Moreau envelope with the parameter $\lambda = 1$  of the maximum function given by
\begin{equation}\label{maxf}
f(x):= \max\big\{x_1,\ldots,x_n\big\}\quad \text{for all }\;x=(x_1,\ldots,x_n)\in\R^n. 
\end{equation}
Due to \cite[Proposition 12.30]{Bauschke}, $e_f (x)$ is continuously differentiable with the Lipschitzian gradient, and \eqref{C11testproblem} belongs to the class of $\mathcal{C}^{1,1}$ optimization problems \eqref{QP}.

\medskip 
The following lemma is needed for implementing our generalized damped Newton algorithm (GDNM) to solve the formulated optimization problem \eqref{C11testproblem}.

\begin{Lemma}[\bf proximal mapping and second-order subdifferentials of maximum functions]\label{prcalofmax} Let $f:\R^n\to\R$ be taken from \eqref{maxf}. Then the proximal mapping of $f$ is calculated by
\begin{equation}\label{proxmax}
\text{\rm Prox}_f(x)=\big(\text{\rm min}\{x_1,s\},\ldots,\text{\rm min} \{x_n,s\}\big),
\end{equation}
where the number $s \in \R$ is such that 
\begin{equation}\label{s-equa}
\sum_{i=1}^n\max\big\{0, x_n-s\big\} =1. 
\end{equation}
For each $x \in \R^n$ and $y \in \partial f(x)$, define the index sets 
$$
J(x):=\big\{i \in \{1,\ldots,n\}\;\big|\; x_i = f(x)\big\}\;\mbox{ and }\;L(x):=\big\{i \in J(x)\;\big|\; y_i >0\big\}.
$$
The calculation formula for the generalized Hessian of $f$ is
\begin{equation}\label{2ndsubofmax}
w \in \partial^2 f (x,y)(v) \iff v_i= c \;\; \forall i \in L(x),\; \sum_{i=1}^n w_i = 0, \; w_i \geq 0\;\; \forall i \in J_{>}(x), \; w_i =0 \;\; \forall i \in J^c(x)\cup J_{<}(x),
\end{equation}
where $c\in \R$ is an arbitrary constant, and where 
$$
J_{>}(x):=\big\{i \in J(x)\;\big|\; v_i > c\big\}, \quad J_{<}(x):=\big\{i \in J(x)\;\big|\; v_i <c\big\}. 
$$
\end{Lemma}
\begin{proof} The proof for the proximal mapping formula \eqref{proxmax} can be found in \cite[Example~6.49]{Beck}. The calculation formula for the second-order subdifferential of the maximum function \eqref{2ndsubofmax} is obtained in \cite[Theorem~3.1]{EH14}. 
\end{proof}

We start implementing Algorithms~\ref{LS} to solve \eqref{C11testproblem} with the explicit computation of their ingredients (gradient and second-order subdifferential of $\varphi$) given entirely via the problem data. More specifically, we need to explicitly determine the gradient $\nabla\varphi(x)$ and the  \textit{coderivative-based Newton direction} $d\in\R^n$ generated by the inclusion 
\begin{equation}\label{newtonincl}
-\nabla \varphi(x)\in\partial^2\varphi(x)(d).
\end{equation}
The expressions for the gradient follows directly from formula \eqref{proxmax} telling us that
\begin{equation}\label{gradientC11testpb}
\nabla\varphi(x) = \nabla e_f(x) + x =  2x-  \text{\rm Prox}_f (x).
\end{equation}

Now we show how to construct $d\in \R^n$ satisfying inclusion \eqref{newtonincl}. For each $x\in \R^n$, define the vector $d$ by
\begin{equation}\label{dkincl}
\displaystyle d_i: = \begin{cases}
\frac{1}{2}\big(c_x - \left(\nabla \varphi(x)\big)_i\right) &\text{if} \quad i \in L(x),\\
\big(- \nabla \varphi(x)\big)_i & \text{otherwise},
\end{cases}
\end{equation} 
where the index set $L(x)$ and the number $c_x$ are given by
$$
L(x):=\big\{i \in \{1,\ldots,n\}\;\big|\;\big(\text{\rm Prox}_f (x)\big)_i =  f\big(\text{\rm Prox}_f(x)\big),\;\big( x - \text{\rm Prox}_f (x)\big)_i >0 big\},
$$ 
$$
c_x:=-\frac{1}{|L(x)|}\sum_{i\in L(x)}\big(\nabla \varphi(x)\big)_i.
$$

The next proposition verifies that the vector $d$ constructed in \eqref{dkincl} indeed satisfies inclusion \eqref{newtonincl}.

\begin{Proposition}\label{findcodNewtond} Let $\varphi:\R^n\to\R$ be given in \eqref{C11testproblem}. For each $x \in\R^n$, consider the vector $d\in \R^n$ defined in \eqref{dkincl}. Then inclusion \eqref{newtonincl} holds for this vector $d$. 
\end{Proposition}
\begin{proof}
Applying the second-order subdifferential sum rule from \cite[Proposition~1.121]{Mordukhovich06} to $\ph$ in \eqref{C11testproblem} gives us
$$
\partial^2\varphi(x)(w) = \partial^2 e_f(x)(w) + w \quad \text{for all }\; w \in \R^n.
$$
This tells us that the inclusion $-\nabla\varphi(x) \in \partial^2\varphi(x)(w)$ is equivalent to
$$
-\nabla \varphi(x) - w \in \partial^2 e_f(x)(w).  
$$
It follows from \cite[Lemma 6.4]{BorisKhanhPhat} that the latter inclusion can be equivalently rewritten as 
\begin{equation}\label{Nmdirection}
-\nabla \varphi(x) - w \in\partial^2 f\big(\text{\rm Prox}_f(x), x - \text{\rm Prox}_f (x)\big)\big(\nabla \varphi(x) + 2w\big). 
\end{equation}  
We only need to show that $w:=d$, for $d$ defined in \eqref{dkincl}, solves \eqref{Nmdirection}. Indeed, fixing any indices $i \in L(x)$ and $j \in \{1,\ldots,n\}\setminus L(x)$ gives us the equalities
\begin{equation}\label{1stconofdk}
\big(\nabla \varphi(x) + 2d\big)_i  =c_x, \quad \big(-\nabla\varphi(x) -d\big)_j = 0,\;\mbox{ and}
\end{equation}
\begin{equation}\label{2ndconofdk}
\begin{array}{ll}
\disp\sum_{i=1}^n\big(-\nabla \varphi(x) - d \big)_i & =\disp\sum_{i \in L(x)}\big(-\nabla \varphi(x) - d\big)_i + \disp\sum_{i \in \{1,\ldots,n\}\setminus L(x)}\big(-\nabla \varphi(x) - d \big)_i \\ 
&=\disp-\frac{1}{2}\sum_{i \in L(x)}\big(c_x + (\nabla\varphi(x))_i\big) =0. 
\end{array}
\end{equation}
Combining \eqref{1stconofdk} and \eqref{2ndconofdk}, we deduce from Lemma~\ref{prcalofmax} that $d$ solves \eqref{Nmdirection} and thus complete the proof. 
\end{proof}

Next we present specifications of Algorithm~\ref{LS} and Theorems~\ref{solvingQP}, \ref{solvingQP2} on its performance for the case of the testing optimization problem \eqref{C11testproblem}.

\begin{Theorem}[\bf solving the testing $\mathcal{C}^{1,1}$ optimization  problem by GDNM]  Consider the testing optimization problem \eqref{C11testproblem}. Then Algorithm~{\rm\ref{LS}}, with all its ingredients calculated in \eqref{gradientC11testpb} and \eqref{dkincl}, either stops after finitely many iterations, or produces a sequence $\{x^k\}$ such that it globally Q-superlinearly converges to $\bar{x}$, which is the unique solution to problem \eqref{C11testproblem}, while being also a tilt-stable local minimizer of 
 the function $\varphi$ from \eqref{C11testproblem}. Furthermore, the sequences $\{\varphi(x^k)\}$ and $\{\nabla\varphi(x^k)\}$ Q-superlinearly converge to $\text{\rm min}\;\varphi$ and $0$, respectively. 
\end{Theorem}
\begin{proof} It is clear that the cost function $\varphi$ in \eqref{C11testproblem} is strongly convex, and thus the generalized Hessian $\partial^2\varphi(x)$ is positive-definite for all $x \in \R^n$. Furthermore,  it follows from \eqref{gradientC11testpb}  and \cite[Proposition~7.4.6]{JPang} that the mapping $\nabla \varphi$ is semismooth on $\R^n$ due to its piecewise linearity.  Applying  now Theorems~\ref{globalconver}, \ref{superlinearLS} and Proposition~\ref{findcodNewtond} to this setting, we arrive at all the conclusions of the theorem.
\end{proof}

Due to Remark~\ref{compareremark1}{\bf(i)}, the underlying difference between our approach and the conventional semismooth Newton method is  in finding the \textit{generalized Newton directions}. To be more specific, the semismooth Newton method to solve \eqref{C11testproblem} is based on using generalized
Jacobian of $\nabla\varphi$ to determine the Newton directions $d$ as solutions to the system of linear equations
\begin{equation}\label{clarkeLiEq} 
-\nabla\varphi(x) = A d, \quad \text{where }\; A \in \partial_C\nabla\ph(x).
\end{equation} 
We have therefore in the setting of \eqref{C11testproblem} that
$$
\partial_C\nabla\ph(x) = 2I -  \partial_C \text{\rm Prox}_f (x)
$$
Consider further the index sets
$$
J(x):=\big\{i\in\{1,\ldots,n\}\;\big|\;\max \{0,x_i-s\} =0\big\},\quad \text{and }\; J^c(x): =\big\{1,\ldots,n\big\}\setminus J(x), 
$$
where $s \in \R$  is such that \eqref{s-equa} holds. It follows from \cite[Example~5.4]{pb14} that an element of $\partial_C \text{\rm Prox}_f (x)$ is the matrix $P$ with the entries
$$
{P_{ij}:= \begin{cases}
1+ 1/(n-|J|) &\text{if}\quad i \ne j, i,j \in J^c(x),\\
1/(n-|J|)&\text{if}\quad i = j, i,j \in J^c(x),\\
1 &\text{if}\quad i,j \in J(x).
\end{cases}}
$$
This tells us that finding a Newton direction $d$ in SNM requires solving the following system of linear equation:
\begin{equation}\label{ClarkeEquation}
-\nabla\varphi(x) = (2I-P)d. 
\end{equation}

Now we are ready to conduct numerical experiments to solve the testing problem \eqref{C11testproblem} by using our algorithm GDNM and the semismooth Newton method SNM. All the numerical experiments are conducted on a desktop with 10th Gen Intel(R) Core(TM) i5-10400 processor (6-Core, 12M Cache, 2.9GHz to 4.3GHz) and 16GB memory. All the codes are written in MATLAB 2016a. We test with different dimensions $n$ ranging from $200$ to $2000.$ The stopping criterion $\left\|\nabla\varphi(x^k)\right\|\le 10^{-6}$ is used in all the tests. The initial points for GDNM and SNM are the same being chosen as a randomly generated vector in $\R^n$ with i.i.d. (identically and independent distributed) standard Gaussian entries. The results are shown in Figure~\ref{fig:my_label}. As we can see from the results presented in this figure, our algorithm  GDNM is highly efficient in solving problem \eqref{C11testproblem}. 
A clear explanation for this is that in the setting under consideration, we can find a precise formula for generalized Newton directions from the second-order subdifferential inclusion \eqref{newtonincl} while SNM requires solving the system of linear equation \eqref{ClarkeEquation}. Typically, although solving the inclusion \eqref{newtonincl} might be more difficult than solving the system of linear equation \eqref{clarkeLiEq}, observe that calculus rules  for the second-order subdifferential $\partial^2 \varphi$ are much more developed than in the generalized Jacobian case $\partial_C \nabla \varphi$ of the semismooth Newton method \eqref{clarkeLiEq}. 
\color{black}
\begin{figure}
\centering
\includegraphics[scale=0.8]{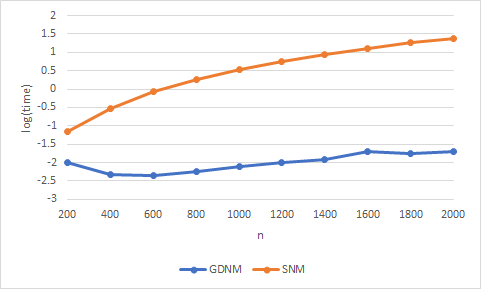}
\caption{CPU time for GDNM and SNM in the $\mathcal{C}^{1,1}$ optimization problem \eqref{C11testproblem}}
\label{fig:my_label}
\end{figure}
\color{black}\vspace*{-0.1in}

\subsection{Solving Lasso Problems}\label{l1 regular}\vspace*{-0.05in}

 {This subsection} is devoted to specifying both Algorithms~\ref{LSQP} and \ref{LSQP2} for the case of the basic Lasso problem stated below, and then to conducting numerical experiments for this problem and comparing them with the performances of some major first-order and second-order algorithms. The {\em basic Lasso problem}, known also as the $\ell^1$-{\em regularized least square optimization problem}, was introduced by Tibshirani \cite{Tibshirani}, and  {since then} it has been largely investigated and applied to various issues in statistics, machine learning, image processing, etc. This problem is formulated as follows:
\begin{eqnarray}\label{Lasso}
\text{minimize }\;\varphi(x):=\frac{1}{2}\|Ax-b\|_2^2+\mu\|x\|_1\quad\text{ subject to }\;x\in\R^n,
\end{eqnarray}
where $A$ is an $m\times n$ matrix, $\mu>0$, and $b\in\R^m$, and where $\|\cdot\|_1$ and $\|\cdot\|_2$ stand for the standard $p$-norms on $\R^n$. It is easy to see that the Lasso problem \eqref{Lasso} belongs to the class of convex composite optimization problems \eqref{QP}. Indeed, we can represent \eqref{Lasso} as minimizing the nonsmooth convex function $\varphi(x):=f(x)+g(x)$, where
\begin{equation}\label{fg}
f(x):=\frac{1}{2}\langle\Tilde{A}x,x\rangle+\langle\Tilde{b},x\rangle+\Tilde{\alpha}\quad\text{and }\;g(x):=\mu\|x\|_1
\end{equation}
with $\Tilde{A}:=A^*A$, $\Tilde{b}:=-A^*b$, and $\Tilde{\alpha}:=\frac{1}{2}\|b\|^2$, and where the matrix $\Tilde{A}=A^*A$ is symmetric and positive-semidefinite. Observe that the Lasso problem \eqref{Lasso} always admits an optimal solution \cite{Tibshirani}.\vspace*{0.05in}

We start implementing Algorithms~\ref{LSQP} and \ref{LSQP2} to solve \eqref{Lasso} with the explicit computation of their ingredients (proximal and subgradient mappings, generalized Hessian of $g$) given entirely via the problem data.

\begin{Proposition}[\bf explicit computations for the Lasso problem] \label{lasso-calc} Let $g(x)=\mu\|x\|_1$ be the regularizer in the Lasso problem \eqref{Lasso}. Then we have the calculation formulas:
\begin{equation}\label{proxofg}
\big(\text{\rm Prox}_{\gamma g}(x)\big)_i=\begin{cases}
x_i-\mu\gamma&\text{if}\quad x_i>\mu\gamma,\\
0&\text{if}\quad-\mu\gamma\le x_i\le\mu\gamma,\\
x_i+\mu\gamma&\text{if}\quad x_i<-\mu\gamma.
\end{cases}
\end{equation}
\begin{equation}\label{first-order}
\partial g(x)=\left\{v\in\R^n\;\bigg|\;
\begin{array}{@{}cc@{}}
v_i=\mu . \text{\rm sgn}(x_i),\;x_i\ne 0,\\
v_i\in[-\mu,\mu],\;x_i=0
\end{array}\right\} \quad\text{whenever }\; x\in\R^n.
\end{equation}
The generalized Hessian of $g$ is calculated by
\begin{equation}\label{second-order}
\partial^2g(x,y)(v)=\Big\{w\in\R^n\;\Big|\;\Big(\frac{1}{\mu} w_i,-v_i\Big)\in G\Big(x_i,\frac{1}{\mu} y_i\Big),\;i=1,\ldots,n\Big\}
\end{equation}
for each $(x,y)\in\text{\rm gph}\,\partial g$ and $v=(v_1,\ldots,v_n)\in\R^n$,
where the mapping $G\colon\R^2\tto\R^2$ is defined by
\begin{equation}\label{G}
G(t,p):=\begin{cases}
\{0\}\times\R&\text{\rm if}\quad t\ne 0,\;p\in\{-1,1\},\\
\R\times\{0\}&\text{\rm if}\quad t=0,\;p\in(-1,1),\\
(\R_{+}\times\R_{-})\cup(\{0\}\times\R)\cup(\R\times\{0\})&\text{\rm if}\quad t=0,\;p=-1,\\
(\R_{-}\times\R_{+})\cup(\{0\}\times\R)\cup(\R\times\{0\})&\text{\rm if}\quad t=0,\;p=1,\\
\emp&\text{\rm otherwise}.
\end{cases}
\end{equation}

\end{Proposition}
\begin{proof} The formula for the proximal mapping \eqref{proxofg} follows from definition \eqref{Prox} and the form of $g(\cdot)=\|\cdot\|_1$. The calculations of $\partial g$ and $\partial^2 g$ are taken from \cite[Propositions~7.1 and 7.2]{BorisKhanhPhat}, respectively.
\end{proof}

Let us present specifications of Algorithms~\ref{LSQP} and \ref{LSQP2} as well as Theorems~\ref{solvingQP} and \ref{solvingQP2} on their performances, respectively, for the Lasso problem \eqref{Lasso}.

\begin{Theorem}[\bf solving Lasso]\label{solveLasso} Considering the Lasso problem \eqref{Lasso}, we have the following:

{\bf(i)} Algorithm~{\rm\ref{LSQP}}, with all its ingredients calculated in Proposition~{\rm\ref{lasso-calc}}, either stops after
finitely many iterations, or produces a sequence $\{x^k\}$ such that it globally Q-superlinearly converges to $\bar{x}$, which is the unique solution to \eqref{Lasso} and a tilt-stable local minimizer of $\ph$ with modulus $\kappa:=1/\lambda_{\text{\rm min}(A^*A)}$, provided that the matrix $A^*A$ is positive-definite.

{\bf(ii)} Algorithm~{\rm\ref{LSQP2}}, with the positive-semidefinite matrix $A^*A$ and the ingredients calculated in Proposition~{\rm\ref{lasso-calc}}, either stops after finitely many steps, or produces a sequence $\{x^k\}$ such that any {accumulation point} $\ox$ of it is a solution to \eqref{Lasso}. If in addition $\partial\varphi$ is metrically regular around $(\bar{x},0)$ with modulus $\kappa>0$, then the sequence $\{x^k\}$ globally Q-superlinearly converges to $\bar{x}$, which is a tilt-stable local minimizer of $\varphi$ with the same modulus.
\end{Theorem}
\begin{proof} Observe by \eqref{first-order} that the graph of $\partial g$ is the union of finitely many closed convex sets, and hence $\partial g$ is semismooth$^*$ on its graph; see \cite{Helmut}. Furthermore, $g$ is a proper, convex, and  {piecewise linear function} on $\R^n$. Then it follows from \cite[Proposition~13.9]{Rockafellar98} that $g$ is twice epi-differentiable on $\R^n$. Applying Theorems~\ref{solvingQP} and \ref{solvingQP2}, we arrive at all the conclusions  in (i) and (ii) of this theorem, respectively.\end{proof}

\color{black}
To run Algorithms~\ref{LSQP} and \ref{LSQP2}, we need to explicitly determine the sequences $\{v^k\}$ and $\{d^k\}$ generated by these algorithms. The expressions for $v^k$ follows directly from formula \eqref{proxofg} telling us that
$$
\left(v^k\right)_i=\begin{cases}
(u^k)_i-\mu\gamma&\text{if}\quad (u^k)_i>\mu\gamma,\\
0&\text{if}\quad-\mu\gamma\le (u^k)_i\le\mu\gamma,\\
(u^k)_i+\mu\gamma&\text{if}\quad (u^k)_i<-\mu\gamma,
\end{cases} \quad \text{where }\; u^k= x^k - \gamma (A^*Ax^k+b). 
$$
Using further the formulas in \eqref{first-order}--\eqref{G}, we express $d^k$ in Algorithm~\ref{LSQP} via the conditions
\begin{equation*}
\begin{cases}
\big(-\frac{1}{\gamma}(x^k-v^k)-A^*Ad^k \big)_i=0&\text{if}\quad\left(v^k\right)_i\ne 0,\\
\big(x^k-v^k + d^k\big)_i=0&\text{if}\quad\left(v^k\right)_i=0.
\end{cases}
\end{equation*}
Thus $d^k$ can be computed for each $k\in\N$ by solving the linear equation $X^k d = v^k-x^k$, where
\begin{equation}\label{X1}
(X^k)_i := \begin{cases}
\gamma(A^*A)_i & \text{if} \quad (v^k)_i \ne 0,\\
I_i  & \text{if}\quad (v^k)_i=0.
\end{cases}
\end{equation}

Similarly, by employing the calculations of Proposition~\ref{lasso-calc} in the framework of Algorithm~\ref{LSQP2} and by performing elementary transformations, we get the linear equation
\begin{equation*}
(BX^k+\gg\mu_k I)d^k=B(v^k-x^k)
\end{equation*}
to find the direction $d^k$, where $X^k$ is computed in \eqref{X1}, and  $B:=I-\gamma A^*A$.\vspace*{0.05in}

Now we are ready to conduct numerical experiments for solving the Lasso problem \eqref{Lasso} by using our globally convergent coderivative-based Generalized Damped Newton Method (GDNM) via Algorithm~\ref{LSQP} and globally convergent coderivative-based  {Generalized Regularized Newton Method (GRNM)}  via Algorithm~\ref{LSQP2}. The obtained calculations are compared with those obtained by implementing the following highly recognized first-order and second-order algorithms:
\begin{itemize}
\item[\bf(i)] The {\em Alternating Direction Methods of
Multipliers \footnote{\href{https://web.stanford.edu/~boyd/papers/admm/lasso/lasso.html}{https://web.stanford.edu/~boyd/papers/admm/lasso/lasso.html}}} (ADMM); see \cite{BPCPE,gabay,glomar}.
\item[\bf(ii)] The {\em Fast Iterative Shrinkage-Thresholing
Algorithm\footnote{\href{https://github.com/he9180/FISTA-lasso}{https://github.com/he9180/FISTA-lasso}}} (FISTA) with the code presented in \cite{BeckTebou}.
\item[\bf(iii)] The {\em Semismooth Newton Augmented Lagrangian
Method\footnote{\href{https://www.polyu.edu.hk/ama/profile/dfsun/}{https://www.polyu.edu.hk/ama/profile/dfsun/}}} (SSNAL) recently developed in \cite{lsk}.
\end{itemize}

 {All the numerical experiments are conducted in the same desktop and software described in Section \ref{testC11experiment}.} All the codes are written in MATLAB 2016a.  In our numerical experiment, $A$ is generated randomly with i.i.d. (identically and independent distributed) standard Gaussian entries,  where $b$ is generated randomly with values of components are from $0$ to $1$. In some particular tests, we normalize each column of $A$ so that $A^*A$ is close to singular and mark them with symbol $^*$ for identification. In summary, $A^*A$ is nonsingular in Tests 3, 4, 7, 8, and it is singular or close to singular in all the other tests. Table \ref{table 2} contains 2 tests where $n>m$ and the matrix $A^*A$ is nonsingular, 2 tests where $n>m$ and the matrix $A^*A$ is singular, 2 tests where $n=m$ and the matrix $A^*A$ is nonsingular, 2 tests where $n=m$ and the matrix $A^*A$ is singular, and 2 tests when $m>n$. To simplify the numerical implementations for solving the Lasso problem \eqref{Lasso}, we set $\mu:= 10^{-3}$ as the tuning parameter for all the tests. If an algorithm cannot start iterating in the first step, it is marked by ``Error" word; this concerns only some cases of GDNM when the matrix $A^*A$ is not positive-definite.  In our numerical experiments, $x^0:=0$ is the starting point for each algorithm, and the following \textit{relative KKT residual} $\eta_k$ in \eqref{KKT} suggested in \cite{lsk} is used to  measure the accuracy of an approximate optimal solution $x^k$ for \eqref{Lasso}:
\begin{equation}\label{KKT}
\eta_k := \frac{\|x^k - \text{\rm Prox}_{\mu\|\cdot\|_1}(x^k-A^*(Ax^k-b)) \|}{1+\|x^k\|+\|Ax^k-b\|}.
\end{equation}
We stop the algorithms when either the condition $\eta_k <10^{-6}$ is satisfied, or the maximum computation time of $10000$ seconds is reached. The results of computations for this part are displayed in Table~\ref{table 2}. There `TN' stands for the test number, `iter' indicates the number of performed iterations, and `CPU time' stands for the time needed to achieve the prescribed accuracy of approximate solutions (the smaller the better).

As we can see from the results presented in Table~\ref{table 2}, our algorithms GDNM and GRNM are highly efficient when $A^*A$ is {\em nonsingular}, where the $Q$-superlinear convergence is guaranteed by Theorem~\ref{solveLasso}. They may behave even better than the other compared algorithms when $m\ge n$, which is the setting of various practically important models; see, e.g., \cite{EHJT04} for Lasso applications to diabetes studies where $m$ is much large than $n$, and \cite{BeckTebou} for $m=n$ with applications to image processing.

When the matrix $A^*A$ is {\em singular} (or close to be singular), our theoretical results do not guarantee the fast convergence of GDNM and GRNM, while the conducted numerical experiments show that GRNM performs better that GDNM and better than FISTA and ADMM in Table~\ref{table 2}, while usually worse than SSNAL. A partial explanation for this is that SSNAL is actually a {\em hybrid} algorithm, which combines the first-order augmented Lagrangian method to solve {\em dual} subproblems, which are strongly convex and of lower dimensions, with the subsequent applications  of the second-order semismooth Newton method. Such a combination exhibits a high efficiency in solving Lasso problems in the singular case.
\begin{sidewaystable}
\begin{minipage}{\textheight}\small
\caption{Solving (\ref{Lasso}) on random instances }\label{table 2}
\newcolumntype{R}{>{\raggedright \arraybackslash} X}
\newcolumntype{S}{>{\centering \arraybackslash} X}
\newcolumntype{T}{>{\raggedleft \arraybackslash} X}
\begin{tabular}{lllllllllllllll}
\hline
\multicolumn{3}{c}{Problem size and TN} & \multicolumn{1}{c}{} & \multicolumn{5}{c}{iter}& \multicolumn{1}{c}{} & \multicolumn{5}{c}{CPU time}\\
\cline{1-3}\cline{5-9}\cline{11-15}
\multicolumn{1}{c}{TN} & \multicolumn{1}{c}{m} & \multicolumn{1}{c}{n} & \multicolumn{1}{c}{} & \multicolumn{1}{c}{SSNAL} & \multicolumn{1}{c}{FISTA} & \multicolumn{1}{c}{ADMM} & \multicolumn{1}{c}{GRNM} & \multicolumn{1}{c}{GDNM} & \multicolumn{1}{c}{} & \multicolumn{1}{c}{SSNAL} & \multicolumn{1}{c}{FISTA} & \multicolumn{1}{c}{ADMM} & \multicolumn{1}{c}{GRNM} & \multicolumn{1}{c}{GDNM}  \\\hline
1  & 400                   & 800                   &  & 25    & 37742 & 22873& 1813 & {Error}   &  & 0.45  & 145.52& 10.89& 45.62& {Error}    \\
2  & 4000                  & 8000                  &  & 153   & 19173 & 19173& 2499 &  {Error}   &  & 847.87& 10000                 & 2359.36                  & 10000                 &  {Error}    \\
3  & 2000                  & 2000                  &  & 43    & 239701& 12785& 59   & 12   &  & 78.38 & 8138.94                   & 158.12                   & 11.07& 2.24  \\
4  & 4000                  & 4000                  &  & 246   & 73374 & 5970 & 59   & 218  &  & 1253.45                   & 10000                   & 320.81                   & 48.16& 178.91\\
5* & 2000                  & 2000                  &  & 22    & 3619  & 90501& 394  & 292  &  & 18.11 & 123.38& 1141.64                  & 65.60& 58.80 \\
6* & 4000                  & 4000                  &  & 24    & 3629  & 103868                   & 520  & 555  &  & 231.40& 462.53& 5166.16                  & 369.27                   & 474.74\\
7  & 800                   & 400                   &  & 4     & 430   & 10   & 6    & 3    &  & 0.14  & 0.86  & 0.02 & 0.11 & 0.08  \\
8  & 8000                  & 4000                  &  & 13    & 487   & 11   & 7    & 3    &  & 18.80 & 117.92& 3.67 & 8.46 & 4.39  \\
9* & 800                   & 400                   &  & 11    & 245   & 426  & 31   & 7    &  & 0.18  & 0.53  & 0.12 & 0.23 & 0.11  \\
10*& 8000                  & 4000                  &  & 11    & 238   & 411  & 72   & 9    &  & 8.37  & 59.18 & 32.17& 56.37& 8.88  \\
\hline
\end{tabular}
\end{minipage}
\end{sidewaystable}
\vspace*{-0.1in}

\subsection{Box Constrained Quadratic Programming}\vspace*{-0.05in}

This subsection is devoted to specifying Algorithms~\ref{LSQP} and \ref{LSQP2} for quadratic programming problems of the form
\begin{eqnarray}\label{QPbox}
\text{minimize }&\;\varphi(x):=\dfrac{1}{2}\langle Ax,x \rangle+\langle b,x\rangle\quad\\
\text{subject to}&\;l_i\le x_i\le L_i \text{ for all }i=1,\ldots,n \nonumber 
\end{eqnarray}
where $A$ is an $n\times n$ positive-semidefinite matrix, and where $b,l,L\in\R^n$ are such that $l_i\le L_i$ as $i=1,\ldots,n$.

\begin{Proposition}\label{calc-QPbox} Considering the indicator function $\delta_\Omega:\R^n\to\overline{\R}$ of the set  
\begin{equation}\label{boxc}
\Omega:=\big\{x\in\R^n\;\big|\;\ell\le x\le L\big\},
\end{equation}
we have the precise calculation formulas
\begin{equation}\label{projOmega}
\text{\rm Prox}_{\gamma g}(x)= P_\Omega(x)=\big(\text{\rm min}\{\text{\rm max}\big\{x_i,\ell_i\big\},L_i\big\}\big)_{i=1}^n\quad\text{for }\;x\in\R^n,\;\gamma>0,
\end{equation} 
\begin{equation}\label{sudOmega} 
\partial\delta_\Omega(x)=N_\Omega(x)=F_1(x_1)\times\ldots \times  F_n(x_n) \quad\text{for }\; x\in\Omega,\;\mbox{ where}
\end{equation} 
$$
F_i(t)=\begin{cases}
(-\infty,0]&\text{if }\quad t =\ell_i,\\
[0,\infty)&\text{if }\quad t = L_i,\\
\{0\} &\text{if }\quad t\in(\ell_i,L_i).
\end{cases}
$$
The generalized Hessian of $\delta_\Omega$ is calculated by
\begin{equation}\label{secorderOmega}
\partial^2\delta_\Omega(x,y)(v)=\big\{w\in\R^n\;\big|\;\big(w_i,-v_i\big)\in G_i\big(x_i,y_i\big),\;i=1,\ldots,n\big\},
\end{equation} 
where the mappings $G_i\colon\R^2\tto\R^2$, $i=1,\ldots,n$, are defined by
\begin{equation}\label{Gi}
G_i(t,p):=\begin{cases}
\R\times \{0\}&\text{\rm if}\quad t = \ell_i,p<0,\\
(\R_{-}\times \R_{+})\cup (\R \times \{0\})\cup (\{0\}\times \R)&\text{\rm if}\quad t=\ell_i,\;p=0,\\
\{0\}\times\R &\text{\rm if}\quad t\in (\ell_i,L_i),\; p =0,\\
(\R_{+}\times\R_{-})\cup(\{0\}\times\R)\cup(\R\times\{0\})&\text{\rm if}\quad t=L_i,\;p=0,\\
\R\times\{0\}&\text{\rm if}\quad t=L_i,\;p>0,\\
\emp&\text{\rm otherwise}.
\end{cases}
\end{equation}
\end{Proposition}
\begin{proof} The formula for the proximal mapping \eqref{projOmega} follows from \cite[Lemma~6.26]{Beck}. Note that 
$$
\delta_\Omega(x)=\delta_{\Omega_1}(x_1)+\ldots +\delta_{\Omega_n}(x_n)\quad \text{for all }\;x = (x_1,\ldots,x_n)\in\R^n,
$$
where $\Omega_i:= [\ell_i, L_i] =\{t\in \R\;|\;\ell_i\le t\le L_i\}$, $i=1,\ldots,n$.  {Using \cite[Exercise~8.14 and Example 6.10]{Rockafellar98}, we obtain \eqref{sudOmega}}. It remains to verify the second-order subdifferential formula \eqref{secorderOmega} for $\delta_\Omega$ at $(x,y)\in\gph\partial\delta_\Omega$. Observe that $
N_{{\rm\small gph}\,\partial\delta_{\Omega_i}} = G_i \quad \text{for all }\;i = 1,\ldots, n$.
This allows us to deduce from \cite[Theorem~4.3]{BorisOutrata} the representation
\begin{equation*}
\partial^2\delta_\Omega(x,y)(v) = \big\{w\in\R^n\;\big|\;\big(w_i,-v_i\big)\in N_{{\rm\small gph}\,\partial\delta_{\Omega_i}}(x_i,y_i),\;i=1,\ldots,n\big\},
\end{equation*}
which therefore justifies the fulfillment of \eqref{secorderOmega} and completes the proof of the proposition.
\end{proof}

Next we obtain specifications of Algorithms~\ref{LSQP} and \ref{LSQP2} together with Theorem~\ref{solvingQP} and \ref{solvingQP2} on their performances, respectively, for the case of box constrained quadratic programming in \eqref{QPbox}.

\begin{Theorem}[\bf solving box constrained quadratic programs]\label{solveQPcon} Considering the quadratic programming problem \eqref{QPbox}, we have the following assertions:

{\bf(i)} Assume that the matrix $A$ is positive-definite. Then  Algorithm~{\rm\ref{LSQP}}, with all its ingredients calculated in Proposition~{\rm\ref{calc-QPbox}}, either stops after
finitely many iterations, or produces a sequence $\{x^k\}$ such that it globally Q-superlinearly converges to $\bar{x}$, which is the unique solution to \eqref{QPbox} and a tilt-stable local minimizer of $\ph$ with modulus $\kappa:=1/\lambda_{\text{\rm min}(A^*A)}$.

{\bf(ii)} Assume that the matrix $A$ is positive-semidefinite. Then Algorithm~{\rm\ref{LSQP2}}, with all its ingredients calculated in
Proposition~{\rm\ref{calc-QPbox}}, either stops after finitely many iterations, or produces a sequence $\{x^k\}$ such that any accumulation point $\ox$ of it is a solution to \eqref{QPbox}. If in addition $\partial\varphi$ is metrically regular around $(\bar{x},0)$ with some modulus $\kappa>0$, then the sequence $\{x^k\}$ globally Q-superlinearly converges to $\bar{x}$, which is a tilt-stable local minimizer of $\varphi$ with the same modulus.
\end{Theorem}
\begin{proof} Reduce \eqref{QPbox} to the equivalent form of convex composite optimization:
$$
\min\quad f(x) + g(x) \quad \text{subject to }\; x\in\R^n,\;\mbox{ where}
$$
$$
f(x):=\frac{1}{2}\langle Ax,x\rangle + \langle b,x\rangle, \quad g(x):= \delta_\Omega(x), \quad\Omega:=\big\{x\in\R^n\;\big|\;\ell\le x\le L\big\}. 
$$
Observe by \eqref{sudOmega} that the graph of $\partial g$ is the union of finitely many closed convex sets, and hence $\partial g$ is semismooth$^*$ on its graph by \cite{Helmut}. Since $\Omega$ is polyhedral, it follows from \cite[Example~10.24]{Rockafellar98} that $g$ is {\em fully amenable} on $\R^n$ in the sense of \cite{Rockafellar98}. Then using \cite[Corollary~13.15]{Rockafellar98} tells us that $g$ is twice epi-differentiable on $\R^n$. Applying now Theorems~\ref{solvingQP} and \ref{solvingQP2}, we verify both assertions of this theorem.
\end{proof}

To run Algorithms~\ref{LSQP} and \ref{LSQP2}, we need to explicitly determine the sequences $\{v^k\}$ and $\{d^k\}$ generated by these algorithms. The expressions for $v^k$ follows directly from \eqref{projOmega}, which tells us that
$$
v^k=\big(\text{\rm min}\big\{\text{\rm max}\{u_i,\ell_i\big\},
L_i\big\}\big)_{i=1}^n.
$$
Using further the formulas in \eqref{sudOmega}--\eqref{Gi}, we express $d^k$ in Algorithm~\ref{LSQP} via the conditions
\begin{equation*}
\begin{cases}
\big(-\frac{1}{\gamma}(x^k-v^k)-Ad^k \big)_i=0&\text{if}\quad\left(u^k-v^k\right)_i= 0,\\
\big(-x^k-v^k + d^k\big)_i=0&\text{if}\quad\left(u^k-v^k\right)_i\ne 0.
\end{cases}
\end{equation*}
Thus $d^k$ can be computed for each $k\in\N$ by solving the linear equation $X^k d = v^k-x^k$, where
\begin{equation}\label{X2}
(X^k)_i := \begin{cases}
\gamma A_i & \text{if} \quad (u^k-v^k)_i = 0,\\
I_i  & \text{if}\quad (u^k-v^k)_i\ne 0.
\end{cases}
\end{equation}
 Similarly, by employing the calculations of Proposition~\ref{calc-QPbox} in the framework of Algorithm~\ref{LSQP2} and by performing elementary transformations, we get the linear equation
\begin{equation*}
(BX^k+\gg\mu_k I)d^k=B(v^k-x^k)
\end{equation*}
to find the direction $d^k$, where $X^k$ is computed in \eqref{X2}, and  $B:=I-\gamma A$.\vspace*{0.05in}
Now we are ready to conduct numerical experiments for solving quadratic programming problems with box constraints by using our GDNM via Algorithm \ref{LSQP} and GRNM via Algorithm \ref{LSQP2}. Our methods are compared with the \textit{trust region reflective algorithm} in MATLAB's quadratic programming solver.  {All the numerical experiments are conducted in the same desktop and software described in Section \ref{testC11experiment}.}

To get the positive-semidefinite matrix $A$, we generate a random $n\times n$ matrix $C$ with i.i.d.\ standard uniform entries and then define $A:=C^*C$. In some particular tests, we put $A:=C^*C/10^7$ so that $A$ is close to be singular and mark these tests by $*$. Then the vectors $b$ and $l$ are generated randomly with i.i.d.\ standard uniform entries. To get the vector $L\in\R^n$ such that $l_i\le L_i$ for all $i=1,\ldots,n$, entries of $L$ are generated independently with uniform distribution on the interval $(1,2)$. The initial points are all ones vector for all the tests and all the algorithms. As suggested in the MATLAB built in quadratic programming solver, the stopping criterion used is the function tolerance one, i.e.,
\begin{align*}
\dfrac{|f(x^k)-f(x^{k+1})|}{1+|f(x^{k})|}\le\varepsilon.
\end{align*}
The tolerance $\varepsilon$ is chosen to be $10^{-9}$ for all the tests and all the algorithms. The results of numerical experiments in this part are shown in Table~\ref{numerical:quad}. In this table, `TR' refers to the {\em trust region reflective method} while other information is the same as in Subsection~\ref{l1 regular}.
\begin{table}[H]
\centering 
\begin{tabular}{lllccccccc} 
\hline 
\multicolumn{2}{c}{TN and size}                   &  & \multicolumn{3}{c}{Iter} & \multicolumn{1}{l}{} & \multicolumn{3}{c}{CPU time}  \\ 
\cline{1-2}\cline{4-6}\cline{8-10}
\multicolumn{1}{c}{TN} & \multicolumn{1}{c}{size} &  & TR & GDNM & GRNM         & \multicolumn{1}{l}{} & TR    & GDNM  & GRNM          \\ 
\hline
1                      & 200                      &  & 6  & 4    & 6            &                      & 0.16  & 0.07  & 0.02          \\
2                      & 500                      &  & 7  & 7    & 6            &                      & 0.08  & 0.05  & 0.04          \\
3                      & 2000                     &  & 7  & 7    & 6            &                      & 1.38  & 1.30  & 2.40          \\
4                      & 5000                     &  & 9  & 8    & 6            &                      & 8.03  & 9.06  & 18.33         \\
5*                     & 200                      &  & 7  & 2    & 5            &                      & 0.61& 0.08  & 0.02          \\
6*                     & 500                      &  & 9  & 3    & 5            &                      & 0.11  & 0.03  & 0.04          \\
7*                     & 2000                     &  & 10 & 10   & 8            &                      & 1.61  & 2.00  & 2.97          \\
8*                     & 5000                     &  & 15 & 39   & 6            &                      & 14.58 & 54.89 & 21.54         \\
\hline
\end{tabular}
\caption{Solving box constrained quadratic programming on random instances}\label{numerical:quad}
\end{table}
The obtained results show that our algorithms GDNM and GRNM are more efficient when the size of the problem is rather small; see, e.g., Tests 1, 2, 3 and 5, 6. It can also be seen that GRNM is more stable than GDNM when the size of the problems is increasing. Tests 4, 7 and 8 indicate that our methods should be further improved to solve problems in high-dimensional spaces.\color{black}

\section{Conclusions and Future Research}\label{sec:conclusion}

In this paper we propose and develop two globally convergent generalized Newton methods to solve problems of ${\cal C}^{1,1}$ optimization and of convex composite optimization with extended-real-valued regularizers, which include nonsmooth problems of constrained optimization. The developed algorithms are far-going extensions of the classical damped Newton method and of the  {regularized Newton} algorithm with the replacement of the standard Hessian by its generalized version applied to nonsmooth (of the second order) functions. {The latter construction} is coderivative generated, which coins the names of our generalized Newton methods. The obtained results demonstrate the efficiently of both algorithms, their global superlinear convergence under appropriate assumptions, and their applications to the solution of Lasso problems   and of box constrained problem of quadratic programming  with conducting numerical experiments.

Our future research includes developing hybrid generalized Newton methods, which contain subproblems that can be efficiently solved by using first-order algorithms, and then combining them with the advanced second-order Newton-type techniques. We intend to establish the global superlinear convergence of iterates under relaxed assumptions that do not involve the positive-definiteness of the generalized Hessian in ${\cal C}^{1,1}$ optimization as well as the strong convexity requirement for problems of convex composite optimization, which will go beyond those with quadratic smooth parts. The obtained results would allow us to develop new applications to Lasso problems as well as to other important classes of models in machine learning, statistic, and related disciplines.\vspace*{-0.1in}

\section*{Appendix: Some Technical Lemmas}\label{sec:Appendix}

This section contains four technical lemmas used in the text.  

\medskip 
The first lemma is a local version of \cite[Lemma~2.20]{Solo14}. The proof of this result is similar to the original one, and thus it is omitted.

\begin{Lemma} \label{estimate1} Let  $\Omega \subset \R^n$ be an open set, and let $\varphi:\Omega \to\R$ be a continuously differentiable function such that $\nabla\varphi$ is Lipschitz continuous with modulus $L >0$. Then for any $\sigma>0$,  $x \in \Omega$, and $d \in \R^n$ satisfying $\langle \nabla \varphi(x),d\rangle <0$, the following inequality 
\begin{equation}\label{est1}
\varphi(x+\tau d)\leq \varphi(x) +\sigma \tau \langle \nabla\varphi(x),d\rangle 
\end{equation}
holds whenever $\tau \in (0,\overline{\tau}]$ and $x+\tau d \in \Omega$, where
$$
\overline{\tau}:= \frac{2(\sigma -1)\langle \nabla \varphi(x),d\rangle}{L\|d\|^2}>0. 
$$
\end{Lemma}

\color{black}
The next lemma provides conditions for the $R$-linear and $Q$-linear convergence of sequences.

\begin{Lemma}[\bf estimates for convergence rates]\label{QRlinear}  Let $\{\alpha_k\}, \{\beta_k \}$, and $\{\gamma_k \}$ be sequences of positive numbers. Assume that there exist numbers $c_i>0$, $i=1,2,3$, and $k_0\in \N$ such that for all $k \ge k_0$ we have the estimates:
\begin{itemize}
\item[\bf (i)] $\alpha_k - \alpha_{k+1} \ge c_1 \beta_k^2$.
\item[\bf (ii)] $\beta_k \geq c_2 \gamma_k$.
\item[\bf (iii)] $c_3 \gamma_k^2 \ge \alpha_k$.
\end{itemize}
Then the sequence $\{\alpha_k \}$ Q-linearly converges to zero, and the sequences $\{\beta_k \}$ and $\{\gamma_k \}$ R-linearly converge to zero as $k\to\infty$.
\end{Lemma}
\begin{proof} Combining (i), (ii), and (iii) yields the inequalities
$$
\alpha_{k} - \alpha_{k+1} \ge c_1 \beta_k^2 \ge c_1 c_2^2 \gamma_k^2 \ge\frac{c_1c_2^2}{c_3}\alpha_{k} \quad \text{for all }\; k \geq k_0,
$$
which imply that $\alpha_{k+1}\le q \alpha_k$, where $q:= 1- (c_1c_2^2)/c_3\in (0,1)$. This verifies that the sequence $\{\alpha_{k}\}$ Q-linearly converges to zero. Using the latter and the assumed condition (i) ensures that
$$
\beta_k^2 \le\frac{1}{c_1}(\alpha_k - \alpha_{k+1}) \le\frac{1}{c_1}\alpha_k \leq \frac{1}{c_1}q\alpha_{k-1} \le \ldots\leq \frac{1}{c_1} q^k \alpha_0 \quad \text{for all }\; k\ge k_0,
$$
which tells us that $\beta_k \le c\mu^k$, where $c:= \sqrt{\alpha_0/c_1}$ and $\mu :=\sqrt{q}$. This justifies the $R$-linear convergence of the sequence $\{\beta_k \}$ to zero. Furthermore, it easily follows from (ii) that the sequence $\{\gamma_k \}$ $R$-linearly converge to zero, and thus we are done with the proof.
\end{proof}

Now we obtain a useful result on the semismooth$^*$ property of compositions.

\begin{Lemma}[\bf semismooth$^*$ property of composition mappings] \label{semismoothcomposite} Let $A \in \R^{n\times n}$ be a symmetric nonsingular matrix, $b \in \R^n$, $\bar{x}\in\R^n$, and $f:\R^n \to\R^n$ be continuous and semismooth$^*$ at $\bar{y}:=A\bar{x}+b$. Then the mapping $g: \R^n \to \R^n$ defined by $g(x):= f(Ax+b)$ is semismooth$^*$ at $\bar{x}$.
\end{Lemma}
\begin{proof} Using the coderivative chain rule from \cite[Theorem~1.66]{Mordukhovich06}, we get
\begin{equation}\label{chainrule}
D^*g(x)(w)= A^* D^*f(Ax+b)(w) \quad \text{for all }\; w \in \R^n.
\end{equation}
Denote $\mu:= \sqrt{\text{\rm max}\{1,\|A\|^2 \}\cdot\text{\rm max}\{1,\|A^{-1}\|^2\}}>0$. Picking any $\epsilon>0$ and employing the semismooth$^*$ property of $f$ at $\bar{y}$, we find $\delta>0$ such that
\begin{equation}\label{semismF}
|\langle x^*, y-\bar{y}\rangle - \langle y^*, f(y) - f(\bar{y})\rangle |\le\frac{\epsilon}{\mu} \big\| (y-\bar{y},f(y)-f(\bar{y}))\big\|\cdot\|(x^*,y^*)\|
\end{equation}
for all $y \in \mathbb{B}_\delta(\bar{y})$ and all $(x^*,y^*) \in\gph D^*f(y)$. Denoting $r:= \delta/ \|A\|>0$ gives us $y:=Ax+b \in \mathbb{B}_\delta(\bar{y})$ whenever $x \in \mathbb{B}_r(\bar{x})$. Picking now $x \in \mathbb{B}_r(\bar{x})$ and $(z^*,w^*) \in\gph D^*g(x)$, we get $(A^{-1}z^*,w^*)\in\gph D^*f(Ax+b)$ due to \eqref{chainrule}. It follows from \eqref{semismF} that
\begin{eqnarray*}
|\langle z^*, x-\bar{x}\rangle - \langle w^*,g(x) - g(\bar{x})\rangle | &=&|\langle A^{-1}z^*, y-\bar{y}\rangle - \langle w^*, f(y) - f(\bar{y})\rangle | \\
&\le& \frac{\epsilon}{\mu}\big\| (y-\bar{y},f(y)-f(\bar{y}))\|\cdot\big\|(A^{-1}z^*,w^*)\big\|\\
&=& \frac{\epsilon}{\mu}\big\|(Ax-A\bar{x},g(x)-g(\bar{x}))\big\|\cdot\big\|(A^{-1}z^*,w^*)\big\|\\
& \le & \epsilon \| (x- \bar{x},g(x)-g(\bar{x}))\|\cdot\|(z^*,w^*)\|,
\end{eqnarray*}
which verifies the semismooth$^*$ property of $g$ at $\bar{x}$.
\end{proof}

The final lemma establishes tilt stability of strongly convex functions at stationary points.

\begin{Lemma}[\bf strong convexity and tilt-stability]\label{strongtilt} Let $\varphi:\R^n\to\overline{\R}$ be an l.s.c. and strongly convex function with modulus $\kappa>0$, and let $\bar{x}\in\dom\varphi$ such that $0\in \partial\varphi(\bar{x})$. Then $\bar{x}$ is a tilt-stable local minimizer of $\varphi$ with modulus $\kappa^{-1}$.
\end{Lemma}
\begin{proof} By the second-order characterization of strongly convex functions \cite[Theorem~5.1]{ChieuChuongYaoYen}, we have
$$
\langle z, w\rangle \ge\kappa \|w\|^2 \quad \text{for all }\; z \in \partial^2\varphi(x,y)(w),\; (x,y) \in\gph\partial\varphi,\;\mbox{and }\;w \in \R^n.
$$
This implies in turn that $\bar{x}$ is a tilt-stable local minimizer with modulus $\kappa^{-1}$ by second-order characterization of tilt stability taken from \cite[Theorem 3.5]{MorduNghia}.
\end{proof}

\section*{Funding and Conflicts of Interests} 

Research of Pham Duy Khanh is funded by the Ministry of Education and Training Research Funding under grant B2023-SPS-02. Research of Boris S. Mordukhovich was partly supported by the US National Science Foundation under grants DMS-1808978 and DMS-2204519, by the US Air Force Office of Scientific Research under grant \#15RT0462, and by the Australian Research Council under Discovery Project DP-190100555. Research of Vo Thanh Phat was partly supported by the US National Science Foundation under grants DMS-1808978 and DMS-2204519, and by the US Air Force Office of Scientific Research under grant \#15RT0462. Research of Dat Ba Tran was partly supported by the US National Science Foundation under grant DMS-1808978 and DMS-2204519.
The authors declare that the presented results are new, and there is no any conflict of interest.

\color{black}

\section*{Acknowledgements} The authors are very grateful to three anonymous referees for their helpful remarks and suggestions, which allowed us to significantly improve the original presentation.
\end{document}